\documentclass[10pt]{article}
\usepackage{amssymb}
\usepackage{amsmath}
\usepackage{theorem}
\usepackage{epsfig}
\usepackage{verbatim}
\usepackage{graphicx}
\usepackage{enumerate}
\usepackage{enumitem}
\usepackage{cancel}
\usepackage{mathtools}
\usepackage{hyperref}
\usepackage{datetime}

\usepackage[square,sort,comma,numbers]{natbib}
\usepackage{color}

\newlength{\hchng}
\newlength{\vchng}

\newcommand {\bc} {\begin{center}}
\newcommand {\ec} {\end{center}}

\makeatletter
\def\keywords{\xdef\@thefnmark{}\@footnotetext}
\makeatother

\theoremstyle{plain}

\newtheorem{thm}{Theorem}[section]
\newtheorem{lem}[thm]{Lemma}
\newtheorem{defn}[thm]{Definition}
\newtheorem{prop}[thm]{Proposition}
\newtheorem{cor}[thm]{Corollary}
\newtheorem{rem}{Remark}[section]
\newtheorem{assum}[thm]{Assumption}
\newtheorem{pro}{Property}
\newtheorem{claim}{Claim}

\allowdisplaybreaks[2]

\newenvironment{proof}[1]{\begin{trivlist} \item[] {\em Proof of #1:}}{\hfill $\Box$
                      \end{trivlist}}

\newcommand{\ud}{\, d}
\newcommand{\la}{\lambda}
\newcommand{\al}{\alpha}

\newcommand{\R}{\mathbb{R}}

\newcommand{\mS}{\mathbb{S}}

\newcommand{\Om}{\Omega}

\newcommand{\pa}{\partial}
\newcommand{\eps}{\epsilon}

\newcommand{\lqu}{\textquoteleft}

\newcommand{\lf}{\left\lfloor}  
\newcommand{\rf}{\right\rfloor}  
\newcommand{\etb}{\emph{\textbf{0}}}
\newcommand{\norm}[1]{\left\lVert#1\right\rVert}

\newcommand{\Ga}{\Gamma}

\title{Two-phase free boundary problems in convex domains}

\date{}     
\author{Thomas Beck \and David Jerison\thanks{The second author was supported in part by NSF Grant 1500771, a Simons Fellowship, Simons Foundation Grant (601948, DJ) and a Guggenheim Fellowship.}  \and Sarah Raynor}
\date{\today}

\begin{document}
\maketitle
\begin{abstract} 
We study the regularity of minimizers of a two-phase free boundary problem. For a class of n-dimensional convex domains, we establish the Lipschitz continuity of the minimizer up to the fixed boundary under Neumann boundary conditions. Our proof uses an almost monotonicity formula for the Alt-Caffarelli-Friedman functional restricted to the convex domain. This requires a variant of the classical Friedland-Hayman inequality for geodesically convex subsets of the sphere with Neumann boundary conditions. To apply this inequality, in addition to convexity, we require a Dini condition governing the rate at which the fixed boundary converges to its limit cone
at each boundary point. 
\end{abstract}

\section{Introduction}
 
 \indent
We dedicate this work to the memory of Eli Stein.   He had a profound impact on each of us in many ways. We are deeply grateful to him for his teaching, guidance and wisdom.

In this paper we study the regularity of minimizers of a two-phase free boundary problem in an $n$-dimensional convex domain. We show that the minimizer is Lipschitz continuous up to the fixed boundary under Neumann conditions by first establishing an almost-monotonicity formula for the Alt-Caffarelli-Friedman (ACF) functional (see \eqref{eq:ACFfunctional}).    This functional is the key to interior regularity, and its properties near a fixed convex boundary under Neumann conditions are of independent interest.  

In our version of almost monotonicity for the ACF functional on convex domains, we require,
in addition to convexity,  a sharp Dini condition controlling the rate at which the boundary converges to its limit cone 
(Assumption \ref{assum:Main}).   This Dini condition holds for every convex domain
in dimension 2, and it is surprisingly hard to construct an example of a convex domain in dimension
3 for which this Dini condition fails (see Section \ref{sec:assum}).  Thus we have not proved Lipschitz regularity on general convex domains.  Instead, we have proved the Lipschitz bound in many special cases and identified a new geometric obstruction to the monotonicity method.   We expect this approach to  lead, ultimately, to a full proof in general convex domains, but such a proof has to circumvent this obstruction.

The variational problem under consideration is as follows. Let $\Omega\subset\mathbb{R}^n$ be an open, 
bounded, convex domain, and let $K\subset \bar \Omega$ be a closed set.   For a given $u_0\in H^1(\bar{\Omega})$, the function $u$ is defined to be the minimizer of the functional 
\begin{equation} \label{eq:ACfunctional}
J[v] = \int_{\Omega} (\left|\nabla v \right|^2 + Q(x) 1_{\{v>0\}}) \, dx
\end{equation}
over the set $v \in  H^{1}(\Omega)$ such that $v = u_0$ on $K$.   Here $Q\in C^\infty(\bar \Omega)$ is a positive, smooth function, and $1_{\{v>0\}}$  is the indicator function of the set $\{v>0\}$.   This variational problem has applications to the flow of two liquids in jet and cavity models (\cite{BZ}, \cite{Gu1}), as well as fluid dynamics, optimal shape design and electromagnetism (see, for example, \cite{Fr1}, \cite{FR}, \cite{CSV}). 

Formally, the Euler-Lagrange equation $J'[u]=0$ is expressed as three equations.
Two of them are the straightforward equations
$\Delta u = 0$ in each of the two phases $\Omega^\pm$, with the positive phase $\Omega^+$ 
defined as $\{x\in \Omega: u(x) >0\}$, and non-positive phase, $\Omega^-$, defined as the interior of the set $\{x\in \Omega: u(x) \le 0\}$.  The third is a gradient jump condition,
\begin{align*}
\left|\nabla u^{+}(x)\right|^2 - \left|\nabla u^{-}(x)\right|^2 = Q(x)
\end{align*}
at points $x$ belonging to the interface between the phases, $ \partial\Omega^+ \cap \partial \Omega^-$,
also known as the free boundary.  The existence of minimizers is easy to establish, but, a priori, they only 
satisfy $\nabla u \in L^2(\Omega)$.  
A major goal of the regularity theory is to show that $u$ does indeed satisfy the jump
condition in a suitable sense. 

The first major step in the regularity theory is to prove that $u$ is Lipschitz continuous. In the one-phase case, defined as the case in which $u\geq0$ and $u\equiv 0$ on $\Omega^-$,
the Lipschitz regularity in the interior was proved by Alt and Caffarelli  \cite{AC1}. To prove interior Lipschitz regularity 
in the two-phase case, Alt, Caffarelli, and Friedman \cite{ACF} introduced the functional $\Phi(t)$ given by
\begin{equation}\label{eq:ACFfunctional}
\Phi(t) = \left(\frac{1}{t^2}\int_{B_t(x_0)}\frac{\left|\nabla u^{+}(x)\right|^2}{|x-x_0|^{n-2}} \ud x \right)\cdot  \left(\frac{1}{t^2}\int_{B_t(x_0)}\frac{\left|\nabla u^{-}(x)\right|^2}{|x -x_0|^{n-2}} \ud x \right),
\end{equation}
where $B_t(x_0)$ is the $n$-dimensional ball centered at an interior point $x_0\in\Omega$. 
They proved that if $\Delta u^{\pm}\geq0$ and $u^{+}u^{-}\equiv 0$, then 
$\Phi(t)$  is a monotone increasing function of $t$, giving control on $\Phi(t)$ for all scales $t$.
This is the key ingredient in the proof that $u$ is Lipschitz continuous in the two-phase case.  Moreover,
the Lipschitz bound is crucial to the subsequent regularity theory of the free
boundary, and, in particular, to the proof that the jump condition is satisfied.  Indeed,
once the solution $u$ is Lipschitz, one can rescale by dilation and study the ``blow-up" limit.
This leads to the further regularity theory of the free boundary in much the way that rescaling to cones 
leads to regularity of area-minimizing surfaces and hypersurfaces.  

We restrict ourselves to the case in which the 
fixed boundary is convex at points where Neumann conditions are imposed.
This is nearly necessary.  Even for harmonic functions away from
a free boundary, Lipschitz regularity up to the fixed boundary in
the presence of Neumann conditions requires a convexity or
exterior ball condition on $\pa\Omega$.  
Moreover, when a free boundary meets the fixed boundary with Dirichlet conditions,
A. Gurevich has shown that the Lipschitz property frequently fails
(see \cite{Gu}). 

In \cite{Ra} the Lipschitz regularity up to the convex Neumann boundary is proved in all dimensions
for the one-phase free boundary problem.  As in the interior case, treated by Alt and Caffarelli, the functional $\Phi$ is not used.  In dimension 2, Lipschitz bounds for two-phase minimizers are proved
in \cite{GMR}.    The present paper addresses the $n$-dimensional case, in which the boundedness of the functional 
$\Phi$, restricted to the convex set, is the most serious new issue.  The proof of the almost-monotonicity of the 
restricted functional $\Phi(t)$ requires a new version of the Friedland-Hayman inequality for eigenvalues on the sphere.   The original inequality concerns the whole sphere.   Our version applies to geodesically convex subsets of the sphere
with Neumann boundary conditions. 

Variants of the monotonicity of the ACF functional $\Phi$ appear in \cite{Ca3} and \cite{FL}. 
Also in \cite{CJK}, an even weaker property, namely that $\Phi$ is bounded, is proved under
under  the weaker assumption that $\Delta u^{\pm}\geq-1$.   While most arguments proving
Lipschitz continuity use the functional $\Phi$, there is an argument due
to Dipierro and Karakhanyan in \cite{DK1} proving interior Lipschitz continuity more 
directly for a functional $J$ with $|\nabla u|^2$ replaced by $|\nabla u|^p$ for $p\neq 2$.   
Their approach is to establish a dichotomy in which either the free boundary is smooth or else the solution $u$ has at most linear growth.

In the next section we state our main results and the Dini condition.  The rest of the paper is structured as follows. In Section \ref{sec:Phi} we prove an almost-monotonicity formula, assuming a Friedland-Hayman inequality for convex
cones and estimates on an error term related to the non-conic character of the convex boundary. This error term is absent in the two-dimensional case. In Section \ref{sec:FH1} we prove the Friedland-Hayman type inequality, by
expressing it in its equivalent formulation in terms of an eigenvalue
problem on geodesically convex subsets of $\mathbb{S}^{n-1}$,
and then deducing it from inequalities on closed manifolds with bounds on Ricci
curvature.  Next we transplant the Friedland-Hayman inequality in Section \ref{sec:FH2} to our nearly convex spherical slices using the Dini condition. In Section \ref{sec:harmonic}, we estimate the error term in the monotonicity formula related to the non-conic character of the boundary. This section uses the exact same Dini condition. These previous sections are then used in Section \ref{sec:Lipschitz} to deduce the Lipschitz continuity of the minimizer from almost-monotonicity. In Section \ref{sec:assum} we show that, in dimension two, Assumption \ref{assum:Main} is in fact satisfied by all convex sets, while there are three (and higher) dimensional convex sets for which it fails. Finally, in the Appendix we provide the proofs of some of the technical results required in Sections \ref{sec:FH2} and \ref{sec:harmonic}, namely, we show how to approximate non-smooth geodesically convex subsets of $\mathbb{S}^{n-1}$ by smooth strictly convex sets and prove properties of Lipschitz parametrizations of the spherical slices.  

\section{Statement of the main results} \label{sec:statement}

We will study the minimizer $u$ near the Neumann boundary, $\pa\Omega\backslash K$.  Fix a ball $B$ such that $\bar B\cap K  = \emptyset$.  Recall that  from
results of \cite{ACF} and \cite{Ra}, the minimizer $u$ of \eqref{eq:ACfunctional}
satisfies
\begin{enumerate}
\item[i)] $u\in C^{\epsilon}(B\cap \bar{\Omega})$ for some $\epsilon>0$;
\item[ii)] $\Delta u = 0$ in $\{u>0\} \cap B \cap \bar\Omega$
and $\{u<0\} \cap B \cap \bar\Omega$;
\item[iii)] the normal derivative $\pa_{\nu}u = 0$ weakly on $B\cap \pa \Omega \cap\{u\neq0\}$.
\end{enumerate}

Here and throughout, $B_r(x)$ will denote the open ball of
radius $r$ in $\R^n$ centered at $x$ and $B_r$ the corresponding
ball centered at the origin.   After suitable dilation and translation,
we may assume the ball $B$ in the preceding paragraph is $B = B_2$.  
\begin{defn} \label{defn:N}
Given $x\in B_1 \cap \pa\Omega$, for $t\in(0,1)$, define the function
\begin{align*}
M_{x}(t)  =  \displaystyle{\sup_{y\in\pa\Omega : 0< |y-x|\le t}} \nu(y)\cdot \left(y-x\right).
\end{align*}
\end{defn}
Here, $\nu(y)$ denotes the outer unit normal to the supporting hyperplane 
of $\Om$ at $y$.  If the normal and support plane are not unique, then
the supremum is taken over all such.

The following condition of Dini type gives a quantitative estimate on
the extent to which $\Om$ is approximated by a cone with vertex at $x$.  
\begin{assum}[Dini Condition for $t^{-1}M_x(t)$] \label{assum:Main}
There exists a constant $C_*$, such that for all $x\in B_1 \cap \pa\Omega$, 
\begin{align*}
\int_{0^+}^{1} \frac{M_{{{x}}}(t)}{t^2} \ud t & < C_*.
\end{align*}
\end{assum}

Note that if $\Om$ is a cone with vertex at $x$, then $M_x(t) \equiv 0$.  
For any convex domain $\Omega$, we have $M_{{x}}(t)\leq Ct$, and 
this bound is just short of ensuring that the assumption holds. 
In Section \ref{sec:assum}, we will discuss this assumption further, and show in particular that it always holds for two dimensional convex sets and for convex $C^{1,\beta}$ domains
for any $\beta>0$.   On the other hand, there are examples of convex 
domains in $\R^n$ for $n\ge 3$ for which the assumption does not hold, 
and we will give a construction in Section \ref{sec:assum}.

Our main theorem states that the minimizer $u$ is Lipschitz continuous up to the Neumann part of the boundary $\pa\Omega$, provided our assumption holds.
\begin{thm} \label{thm:Lipschitz} 
 Suppose that $\Om$ is a convex domain, and let $u$ be a minimizer
of \eqref{eq:ACfunctional} among functions equal to $u_0$ on the compact set $K$.
After suitable dilation and translation, we may assume $0 \in \pa\Om$,
$\Om$ has diameter at least $2$, and $\bar B_2 \cap K = \emptyset$.  
 Under \emph{Assumption \ref{assum:Main}}, there exists a constant $C$, depending only on the $C^{\eps}$-norm of $u$ in $B_2 \cap \bar \Om$, 
 the Lipschitz constant of $\pa\Omega$, and the constant $C_*$ appearing in the  Dini condition, such that 
 \begin{align*}
 \sup_{\Omega\cap B_{1}}\left|\nabla u\right|\leq C.
 \end{align*}
\end{thm}

To prove Theorem \ref{thm:Lipschitz}  
a key ingredient is to establish an almost-monotonicity formula for
 the ACF functional of \cite{ACF} restricted to the
 convex domain.  
More precisely, let $x \in {\Omega}\cap B_1$,
and let  $u = u^+-u^-$, with $u^{\pm}\geq0$.   For 
$0 < t \leq 1$, we then define the function $\Phi(t)$ by
\begin{align} \label{eqn:Phi}
\Phi(t) = \left(\frac{1}{t^2}\int_{B_t(x) \cap\Omega}\frac{\left|\nabla u^{+}(y-x)\right|^2}{|y-x|^{n-2}} \ud y \right)\cdot  \left(\frac{1}{t^2}\int_{B_t(x)\cap\Omega}\frac{\left|\nabla u^{-}(y-{x})\right|^2}{|y-x|^{n-2}} \ud y \right).
\end{align}
A key step in the proof of Theorem \ref{thm:Lipschitz} is  a uniform upper bound on $\Phi(t)$.
\begin{thm} \label{thm:Phi}
Under \emph{Assumption \ref{assum:Main}}, we have
\begin{align*}
\Phi(t) \leq C
\end{align*} 
for $0<t\leq 1$, with a constant $C$ depending only on  the $C^{\epsilon}$-norm of $u$, the Lipschitz constant of $\pa\Omega$, and the constant $C_{*}$ appearing in the Dini condition. 
\end{thm}

One important ingredient of the proof of  Theorem \ref{thm:Phi} is 
the appropriate variant of the Friedland-Hayman inequality, which we now
state.  
\begin{thm} \label{thm:FrHa}  Let $\Ga$ be an open convex cone in $\R^n$.  Suppose that $\Ga_i$,
$i = 1,\, 2$, are open cones such that $\Ga_1 \cap \Ga_2 = \emptyset$,
$\Ga_1 \cup \Ga_2 \subset \Ga$. Let 
$u_i$, $i= 1, \, 2$, be the (unique up to positive multiples)
positive, harmonic functions on $\Ga_i$ 
satisfying mixed boundary conditions, 
$u_i = 0 $ on $\Ga \cap \pa\Ga_i$ and $\pa_\nu u_i = 0$ on $\pa \Ga \cap \pa \Ga_i$.
If the degree of homogeneity of $u_i$ is $\alpha_i$, then
\[
\alpha_1 + \alpha_2 \ge 2.
\]
\end{thm}
We will prove this theorem in Section \ref{sec:FH1} in the equivalent
form on the sphere, Theorem \ref{thm:convex}.

\section{A differential inequality for the functional $\Phi(t)$} \label{sec:Phi}

The main differential inequality we will prove is as follows.
\begin{prop} \label{prop:Main}
Under \emph{Assumption \ref{assum:Main}}, there exists a constant $C_1$ depending only on the Lipschitz constant of $\pa\Omega$ and the constant $C_*$ in the Dini condition, such that for all $s\in(0,1)$, 
\begin{align*}
\int_{s}^1\frac{\Phi'(t)}{\Phi(t)}\ud t  \geq - C_1.
\end{align*}
\end{prop}
Since $\Phi$ is non-negative and $\Phi(1)$ is bounded, Theorem \ref{thm:Phi} follows immediately from this proposition.

To prove this proposition, we follow the approach to estimating $\Phi'/\Phi$ 
of Alt, Caffarelli and Friedman \cite{ACF} in order to reduce the proof 
to our new version of the Friedland-Hayman inequality
for geodesically convex subsets of the sphere, Theorem \ref{thm:convex}.  When we carry
this out we find two major differences.
First of all, there is an additional boundary term $A^\pm(t)$ on $B_t \cap \partial \Omega$  (see Definition \ref{defn:A-plus}). 
Secondly,  because the spherical cross-section
$\Omega \cap \partial B_t$ need not be convex,
the corresponding eigenvalues $\lambda^\pm(t)$ (see Definitions \ref{defn:V} and \ref{defn:DN}) will only
satisfy an approximate form of the Friedland-Hayman inequality.
This means that additional estimates are required to establish
Proposition \ref{prop:Main}.   These estimates, Propositions \ref{prop:Main2}, \ref{prop:Main1}, and \ref{prop:Main3}, will be formulated in this section 
and proved in subsequent sections using Assumption \ref{assum:Main}.

To begin, for almost every $t\in(0,1)$, we have
\begin{align} \label{eqn:Phi-deriv1}
\frac{\Phi'(t)}{\Phi(t)} = \frac{\displaystyle \int_{\pa B_t\cap\Omega}\frac{ \left|\nabla u^{+}\right|^2}{|x|^{n-2}} \ud \sigma}{\displaystyle \int_{B_t\cap\Omega}\frac{ \left|\nabla u^{+}\right|^2}{|x|^{n-2}} \ud x} + \frac{\displaystyle  \int_{\pa B_t\cap\Omega}\frac{\left|\nabla u^{-}\right|^2}{|x|^{n-2}} \ud \sigma}{\displaystyle \int_{B_t\cap\Omega}\frac{\left|\nabla u^{-}\right|^2}{|x|^{n-2}} \ud x} - 4t^{-1}.
\end{align}
Next, we will show that a computation analogous to \cite{ACF},  yields
\begin{align} \label{eqn:Phi-deriv3}
\int_{B_t\cap\Omega} \frac{|\nabla u^{+}|^2}{|x|^{n-2}} \ud x & \leq  t^{-(n-2)}\int_{\pa B_t\cap \Omega}u^{+} \frac{x}{|x|}\cdot\nabla u^{+} \ud \sigma + \tfrac{n-2}{2}\left(1 + A^{+}(t)\right) t^{-(n-1)}\int_{\pa B_t\cap\Omega} (u^{+})^2 \ud \sigma,
\end{align}
along with an analogous inequality for $u^-$. The new term $A^+(t)$, not present in the calculation of \cite{ACF}, is defined as follows.
\begin{defn} \label{defn:A-plus}
For each $t\in(0,1)$ with $\int_{\pa B_t\cap\Omega} (u^{+})^2 \ud \sigma\neq0$, define $A^{+}(t)$ to be the quotient of boundary integrals given by
\begin{align*}
A^{+}(t) = \frac{ \displaystyle{\int_{B_t\cap\pa\Omega} (u^{+})^2\frac{\nu(x)\cdot x}{|x|^{n}} \ud \sigma}}{\displaystyle{t^{-(n-1)}\int_{\pa B_t\cap\Omega} (u^{+})^2 \ud \sigma}}.
\end{align*}
Formally extend $A^{+}(t)$ to be defined on $(0,1)$ by setting $A^{+}(t) =0$ whenever the numerator vanishes, and $A^{+}(t) = \infty$ if only the denominator vanishes. The function $A^{-}(t)$ is defined similarly.
\end{defn}

To prove \eqref{eqn:Phi-deriv3}, use the subharmonicity of the two phases $u^{\pm}$, the weak Neumann boundary condition, and integrating by parts, to obtain
\begin{align*}
\int_{B_t\cap\Omega} \frac{|\nabla u^{+}|^2}{|x|^{n-2}} \ud x & \leq \frac{1}{2}\int_{B_t\cap\Omega}\frac{\Delta\left((u^{+})^2\right)}{|x|^{n-2}} \ud x 
 = \frac1{t^{n-2}}\int_{\pa B_t\cap \Omega}u^{+} \frac{x}{|x|}\cdot\nabla u^{+} \ud \sigma + \tfrac{n-2}{2}\int_{B_t\cap\Omega}\frac{\nabla\left((u^{+})^2\right)\cdot x}{|x|^n} \ud x ,
\end{align*}
and  the similar inequality involving  $u^{-}$. Integrating by parts again in the last integral gives
\begin{align}  \label{eqn:Phi-deriv2}
\int_{B_t\cap\Omega} \frac{|\nabla u^{+}|^2}{|x|^{n-2}} \ud x  
\leq  
\frac1{t^{n-2}}\int_{\pa B_t\cap \Omega}u^{+} \frac{x}{|x|}\cdot\nabla u^{+} \ud \sigma 
+ 
\tfrac{n-2}{2t^{n-1}} \int_{\pa B_t\cap\Omega} (u^{+})^2 \ud \sigma 
 + \tfrac{n-2}{2} \int_{B_t\cap\pa\Omega} (u^{+})^2\frac{\nu(x)\cdot x}{|x|^{n}} \ud \sigma,
\end{align}
where we have used $\int_{B_t\cap\Omega} (u^{+})^2\Delta\left(|x|^{-(n-2)}\right) \ud x \le 0$.  Finally, rewriting \eqref{eqn:Phi-deriv2} in terms of $A^+(t)$,
we obtain \eqref{eqn:Phi-deriv3} as claimed.

Define $B^{+}(t)$ by
\begin{align} \label{eqn:B-plus}
B^{+}(t) = \frac{\displaystyle \int_{\pa B_t\cap\Omega} \left(\pa_{\rho} u^{+}\right)^2 + t^{-2}\left|\nabla_{\theta} u^{+}\right|^2 \ud \sigma }{\displaystyle \int_{\pa B_t\cap \Omega}u^{+}\pa_{\rho}u^{+} \ud \sigma + \tfrac{n-2}{2}\left(1 + A^{+}(t)\right) t^{-1}\int_{\pa B_t\cap\Omega} (u^{+})^2 \ud \sigma},
\end{align}
with $\rho$ the radial variable in spherical polar coordinates, and $\nabla_{\theta}$  the gradient on $\mS^{n-1}$. Inserting \eqref{eqn:Phi-deriv3} into the expression for $\Phi'(t)/\Phi(t)$ from  \eqref{eqn:Phi-deriv1} gives 
\begin{align} \label{eqn:Phi-deriv4}
\frac{\Phi'(t)}{\Phi(t)} & \geq B^{+}(t) + B^{-}(t) - 4t^{-1} .
\end{align}
To obtain a lower bound on \eqref{eqn:Phi-deriv4} we need to control some Dirichlet-Neumann eigenvalues on $\mS^{n-1}$.  We first make the following definitions:
\begin{defn}  \label{defn:V}
For each $t\in(0,1)$, let $V_{t}$ be the set in $\mS^{n-1}$ obtained from intersecting $\Omega$ with the sphere of radius $t$ centred at the origin, and then rescaling the resulting set to the unit sphere. That is,  the set $V_t$ is given by
\begin{align*}
V_{t} = t^{-1}\left(\Omega\cap t\mathbb{S}^{n-1}\right).
\end{align*}
\end{defn}
Note that since $\Omega$ is convex, for all $0<s<t<1$ we have the inclusion $V_t\subset V_{s}$, and we define the limiting domain as $V_0:=\bigcup_{t>0} V_t$. The sets $V_{t}\subset \mathbb{S}^{n-1}$ are not necessarily geodesically convex subsets of the sphere for $t>0$, but $V_0$ is geodesically convex. This is because the extension of $V_{0}$ to a cone with vertex at $\textbf{0}$ is a Euclidean convex subset of $\mathbb{R}^n$.  Denoting $u_t$ by $u_t(x) = u(tx)$  to be a rescaling of $u$, we define the sets
\begin{align} \label{eqn:Ui}
U_t^{+} = \{u_t>0\}\cap V_t, \quad U_t^{-} = \{u_t<0\}\cap V_t, \quad \gamma^{+}_t = \pa\{u_t>0\}\cap {V}_t, \quad \gamma^{-}_t = \pa\{u_t<0\} \cap {V}_t
\end{align} 
for each $t\in(0,1)$.  We thus have $ \bar{U}_t^{+}\cup\bar{U}_t^{-}\subset V_t$ for disjoint sets $U_t^{+}$, $U_t^{-}$, and the Dirichlet-Neumann eigenvalues of interest are the following:
\begin{defn} \label{defn:DN}
We define $\la^{+}(t)$ to be the first Dirichlet-Neumann eigenvalue for $U_{t}^{+}\subset\mS^{n-1}$, with Dirichlet boundary conditions on $\gamma^{+}_t$, and Neumann boundary conditions on the rest of $\pa U_{t}^{+}$. That is, 
\begin{align*}
\la^{+}(t) = \inf_{\phi\in X_t^{\pm}} \frac{\int_{U_t^{+}} \left|\nabla_g\phi\right|^2 \ud\sigma}{\int_{U_t^{+}} \phi^2\ud\sigma},
\end{align*}
where $g$ is the round metric on $\mS^{n-1}$. The set of functions $X^{+}_t$ is given by $X_t^{+} = \{\phi \in C^{\infty}(U_t^{+}) : \emph{supp}(\phi) \cap \gamma^{+}_t =\emptyset \}.$ 
 We also define the characteristic exponent $\al^{+}(t)$ to be the positive solution of \begin{align} \label{eqn:DN1} \al^{+}(t)^2 + (n-2)\al^{+}(t) -\la^{+}(t)=0, \qquad \al^{+}(t) = -\frac{n-2}{2}+\sqrt{\frac{(n-2)^2}{4} + \la^{+}(t)} . 
 \end{align}
 The quantities $\la^{-}(t)$ and $\al^{-}(t)$ are defined likewise. 
\begin{rem} \label{rem:char1}
If $v^{+}_{t}(x)$ is the corresponding first Dirichlet-Neumann eigenfunction for $U_t^{+}$, the characteristic exponent $\al^{+}(t)>0$ is the homogeneity of the unique homogeneous harmonic extension of $v^{+}_t(x)$ to the cone generated by $U_t^{+}$ with vertex at $\emph{\textbf{0}}$. 
\end{rem}
\end{defn}
Next, we follow the same method as in the proof of Lemma 5.1 in \cite{ACF} to obtain lower bounds on $B^\pm(t)$.  It suffices, by homogeneity, to consider the case $t = 1$.   Setting
\[
z = \int_{\Omega \cap \pa B_1} (\pa_\rho u^+)^2 \, d\sigma,
\qquad
w = \int_{\Omega \cap \pa B_1} \left|\nabla_\theta  u^+\right|^2 \, d\sigma,
\qquad s = \sqrt{z/w},
\]
we claim that 
\begin{equation}\label{eq:Bplus1}
B^+(1) \ge \frac{ z + w}{\sqrt{\frac{zw}{\lambda} }+ \frac{aw}\lambda}
= \sqrt{\lambda}\frac{s^2 + 1}{s + \frac{a}{\sqrt\lambda}} \, ,
\end{equation}
with 
\[
a = \frac{n-2}2 \left(1 + A^+(1)\right), \qquad \lambda = \lambda^+(1).
\]
Indeed, the numerator of $B^+(1)$ is $z + w$ and the denominator
is bounded above by $\sqrt{zw/\lambda} + aw/\lambda$ using the Cauchy-Schwarz
inequality and the eigenvalue bound
\[
\frac{w}{\lambda} \ge \int_{\Omega \cap \pa B_1} (u^+)^2 \, d\sigma \,. 
\]
Next, using routine calculus or by completing the square, one finds
that 
\[
\sqrt{\lambda}\, (s^2 + 1) \ge 2\beta \left(s + \frac{a}{\sqrt\lambda}\right),
\]
with $ \beta = -a + \sqrt{a^2 + \lambda}$, the positive root of 
$\beta^2 + 2a \beta - \lambda   = 0$.   Thus, using \eqref{eq:Bplus1}, and reintroducing the radius $t$,  we have the following lower bounds.
\begin{prop} \label{prop:Main-half}  For each $t\in(0,1)$, 
let $\beta^+(t)$ be 
positive root of the equation
\begin{align} \label{eqn:Phi-deriv6}
\beta^{+}(t)^2 + (n-2)(1+A^{+}(t))\beta^{+}(t)  - \lambda^{+}(t) = 0,
\end{align}
with $\beta^+(t)= 0$ when $A^+(t) = \infty$, and similarly for $\beta^-(t)$.  Then 
\begin{align} \label{eqn:Phi-deriv6a}
B^\pm(t) \ge 2 t^{-1} \beta^\pm(t) \quad \mbox{and} \qquad
\frac{\Phi'(t)}{\Phi(t)} \geq 2t^{-1}\left(\beta^{+}(t) + \beta^{-}(t) - 2\right),
\end{align} 
In particular $\Phi'(t)/\Phi(t)\geq -4t^{-1}$ for all $t\in(0,1)$.  
\end{prop}
To prove Proposition \ref{prop:Main}, we now want a lower bound on $\beta^{+}(t)$ and $\beta^{-}(t)$.  
Comparing the equations in \eqref{eqn:DN1} and \eqref{eqn:Phi-deriv6}, we have
\begin{align} \label{eqn:Phi-deriv7}
\left|\alpha^{+}(t)-\beta^{+}(t)\right| \leq CA^{+}(t).
\end{align}
Therefore, a key ingredient in the proof of Proposition \ref{prop:Main} (and hence Theorem \ref{thm:Phi}) is to establish a Friedland-Hayman type inequality \cite{FH}, for the characteristic exponents $\al^{+}(t)$, $\al^{-}(t)$ from Definition \ref{defn:DN}:
\begin{prop} \label{prop:Main2}
Provided \emph{Assumption \ref{assum:Main}}, the Dini condition, holds, the characteristic exponents $\al^{+}(t)$, $\al^{-}(t)$ from Definition \ref{defn:DN} satisfy 
\begin{align*}
\int_0^1 \frac{\left[2-\al^{+}(t) - \al^{-}(t)\right]_{+}}{t} \ud t \leq C_2,
\end{align*}
for a constant $C_2$ depending on the Lipschitz norm of $\pa\Omega$, and the constant $C_*$ in the Dini condition.  
\end{prop}
We will prove Proposition \ref{prop:Main2} in Section \ref{sec:FH2}, using the lower bounds on Dirichlet-Neumann eigenvalues for geodesically convex subsets of the sphere established in Section \ref{sec:FH1}. Combining the estimates in Propositions \ref{prop:Main-half} and \ref{prop:Main2} with \eqref{eqn:Phi-deriv7}, if we could show that
\begin{align*}
\int_{0^+}^{1}\frac{A^{\pm}(t)}{t} \ud t
\end{align*}
is bounded then Proposition \ref{prop:Main} follows.  Instead, we 
will prove this for  the ``important" values of $t$ in Proposition \ref{prop:Main3}
and handle the exceptional values
differently. 
\begin{defn} \label{defn:Sc}  We  broaden the definition of $M_{\etb}(t)$ from \eqref{defn:N} to include the  case of base points $\etb \notin\pa\Omega$, by setting $M_{\etb}(t) = 0$ for $0<t<\emph{dist}(\etb,\pa\Omega)$. For $c>0$ fixed, define the set $S_c$ by
\begin{align*}
S_{c} = \{t\in(0,1):s^{-1}M_{\etb}(s) \leq c \emph{ for all } s\in[\tfrac{1}{4}t,4t]\}.
\end{align*}
\end{defn}
Assumption \ref{assum:Main} controls the size of the set $(0,1)\backslash S_c$ 
in the following sense.
\begin{lem} \label{lem:Sc}
Under Assumption \ref{assum:Main}, there exists a constant $C$, depending only on $c$, and the constant $C_*$ in the Dini condiition such that
\begin{align*}
\int_{(0,1)\backslash S_c}t^{-1}\ud t < C.
\end{align*}
\end{lem}
\begin{proof}{Lemma \ref{lem:Sc}}
We first show that the estimate in Assumption \ref{assum:Main} continues to hold for $\textbf{0}\notin \pa\Omega$: Let $z\in \pa\Omega$, with $z = \text{dist}(\textbf{0},\pa\Omega)$. Then, for any $y\in \pa\Omega$, we write
\begin{align*}
\nu(y)\cdot y = \nu(y)\cdot(y-z) + z\cdot \nu(y),
\end{align*}
so that $M_{\textbf{0}}(t) \leq M_{\textbf{z}}(2t) + |z|$ for $t\geq |z|$, and  $M_{\textbf{0}}(t) =0$ for $t<|z|$. By Assumption \ref{assum:Main} applied to $z\in\pa\Omega$, this ensures that $t^{-2}M_{\textbf{0}}(t)$ is integrable. Now suppose that $t\in (0,1)\backslash S_{c}$, so that $s^{-1}M_{\textbf{0}}(s)>c$ for some $s\in[\tfrac{t}{4},4t]$. Then, since $M_{\textbf{0}}(s)$ is increasing in $s$, we have $t^{-1}M_{\textbf{0}}(t)>\tfrac{c}{4}$. Therefore,
\begin{align*}
\frac{c}{4}\int_{(0,1)\backslash S_c}t^{-1} \ud t \leq \int_{(0,1)\backslash S_c}\frac{M_{\textbf{0}}(t)}{t^2}  \ud t
\end{align*}
Since the right hand side is bounded by $C^*$ by our assumption, this gives the desired estimate.
\end{proof}
Combining Lemma \ref{lem:Sc} with the estimate $\Phi'(t)/\Phi(t)\geq-4t^{-1}$ from Proposition \ref{prop:Main-half}, we see that to prove Proposition \ref{prop:Main} we can pick $c>0$ depending on the Lipschitz constant of $\Omega$ and restrict our estimates to $t\in S_c$. When $c$ is sufficiently small and $t\in S_c$, the spherical slices $V_t$ inherit some regularity properties allowing us to estimate $A^{+}(t)$ and $A^{-}(t)$.
\begin{prop} \label{prop:Main1}
There exists $c>0$, depending on the Lipschitz constant of $\pa\Omega$, such that under \emph{Assumption \ref{assum:Main}}, for all $t\in S_c$, the function $A^{+}(t)$ satisfies
\begin{align*}
A^{+}(t) \leq C_3,
\end{align*}
for a constant $C_3$ depending only on the Lipschitz constant of $\pa\Omega$ and the constant $C_*$ in the assumption.
\end{prop}
We will prove Proposition \ref{prop:Main1} in Section \ref{sec:harmonic} by first obtaining estimates on harmonic measure for a domain with mixed Dirichlet-Neumann boundary conditions. However, inserting the estimate of $A^{\pm}(t)\leq C_3$ for $t\in S_c$ into \eqref{eqn:Phi-deriv7} does not give a strong enough lower bound on $\Phi'(t)/\Phi(t)$. To obtain an alternative estimate on $A^{\pm}(t)$, we use a dyadic decomposition. 
\begin{defn} \label{defn:dyadic}
Given $c>0$, write $S_c  = {\bigcup_{j\in J}I_j}$. Here $J = J(c)$ is a subset of the natural numbers, and $I_j = [t_j,t'_j)$ is the intersection of $[2^{-j},2^{-j+1})$ with $I$.
\end{defn}
Note that from Definition \ref{defn:Sc}, if $t\notin S_c$ then either $[t,2t]\subset (0,1)\backslash S_c$ or $[\tfrac{1}{2}t,t]\subset (0,1)\backslash S_c$, and so $I_j$ is guaranteed to be an interval. Using this decomposition, we will prove the following.
\begin{prop} \label{prop:Main3}
Let $J_1 = J_1(c)$ be the set of $j\in J(c)$ for which $\Phi(t_j) \geq \Phi({t'_{j}})$. Then, provided Assumption \ref{assum:Main}, the Dini condition, holds, there exist constants $c>0$ and $C_4$ such that
\begin{align*}
\sum_{j\in J_1(c)}\int_{I_j} \frac{A^{\pm}(t)}{t} \ud t \leq C_4.
\end{align*}
Here $c$ depends only on the Lipschitz constant, and $C_4$ additionally on the constant $C_*$ in the Dini condition. 
\end{prop}
Roughly speaking, the reason for introducing the intervals $J_1$ is as follows. For each interval $I_j$, one case is that for most $t \in I_j$ the measure of one of the spherical slices $U_{t}^{+}$ or $U_t^{-}$ is small, leading to a large characteristic exponent $\beta^{+}(t)$ or $\beta^{-}(t)$. Hence from Proposition \ref{prop:Main-half}, $\Phi(t)$ will increase on this interval and $j\notin J_1$. Alternatively, if the spherical slices $U_t^{\pm}$ are not small, then we will be able to obtain an improved estimate on $A^{\pm}(t)$ than in Proposition \ref{prop:Main1}. We will prove Proposition \ref{prop:Main3} in Section \ref{sec:harmonic}. 
\\
\\
For any $j\in J\backslash J_1$, we have $\int_{I_{j}} \Phi'(t)/\Phi(t) \ud t \geq0$, and so inserting the estimates from Propositions \ref{prop:Main2} and \ref{prop:Main3} into Proposition \ref{prop:Main-half} ensures that $\Phi'(t)/\Phi(t)$ satisfies the desired lower bound. Thus to complete the proof of Proposition \ref{prop:Main} (and hence also Theorem \ref{thm:Phi}), it is sufficient to prove Propositions \ref{prop:Main2}, \ref{prop:Main1}, and \ref{prop:Main3}.

\section{A Friedland-Hayman inequality for convex subsets of the sphere} \label{sec:FH1}

In this section we will assume $V_{t}$ is a geodesically convex subset of $\mS^{n-1}$
and prove the sharp Friedland-Hayman type inequality 
$\alpha^+(t) + \alpha^-(t)\ge 2$, with no additional error term.
We will use this in the next section to prove the weaker lower bound 
given by Proposition \ref{prop:Main2}. 
\begin{thm} \label{thm:convex}
Let $W$ be a closed, proper, convex  subset of $\mS^{n-1}$. Suppose that there exists a H\"older continuous function $w\in C^{\alpha}(W)$ such that
\begin{align*}
W^{+} = \{w>0\}, \qquad W^{-} = \{w<0\}, \qquad \gamma^{+} = \pa\{w>0\}\cap \mathring{W}, \qquad \gamma^{-} = \pa\{w<0\}\cap\mathring{W},
\end{align*}
with $\bar{W}^{+}\cup\bar{W}^{-}\subset W$, $\left(\pa W^{\pm}\right)\backslash\gamma^{\pm} \subset \pa W$. Define $\mu\left(W^{\pm}\right)$ to be the first Dirichlet-Neumann eigenvalue on $W^{\pm}$, with Dirichlet boundary conditions on $\gamma^{\pm}$, and Neumann boundary conditions on the rest of $\pa W^{\pm}$, and set 
\begin{align*}
\alpha\left(W^{\pm}\right) = -\frac{n-2}{2}+ \sqrt{\frac{(n-2)^2}{4} + \mu\left(W^{\pm}\right)}
\end{align*}
to be the corresponding characteristic exponents. Then, the characteristic exponents satisfy the lower bound
\begin{align*}
\alpha\left(W^{+}\right)  + \alpha\left(W^{-}\right)  \geq 2.
\end{align*} 
\end{thm}
In particular, letting $\Gamma\subset\R^n$ be the convex cone generated by $W$, with vertex at the origin, and $\Gamma_1$, $\Gamma_2\subset \R^n$ be the cones generated by $W^+$, $W^-$ respectively, the characteristic exponents $\alpha(W^{+})$, $\alpha(W^{-})$ are precisely the degrees of homogeneity $\alpha_1$, $\alpha_2$ from Theorem \ref{thm:FrHa}. Therefore, Theorem \ref{thm:convex} implies the desired estimate in Theorem \ref{thm:FrHa}. In this section, we will prove Theorem \ref{thm:convex} under an extra smoothness assumption.
\begin{assum}  \label{assum:smooth}
The subset $W$ of $\mS^{n-1}$ in Theorem \ref{thm:convex} is smooth and strictly convex. The sets $\gamma^{\pm}$ are $(n-2)$-dimensional submanifolds, smooth up to the boundary of $W$. 
\end{assum}
\begin{rem} \label{rem:strict-convex}
The form of strict convexity that we use for $W\subset \mS^{n-1}$ in Assumption \ref{assum:smooth} is that the second fundamental form is strictly positive definite at each point on the boundary $\pa W$.
\end{rem}
In the Appendix we will carry out an approximation argument in order to prove Theorem \ref{thm:convex} without Assumption \ref{assum:smooth}.
\begin{rem} \label{rem:convex}
The lower bound of $2$ in \emph{Theorem \ref{thm:convex}} cannot be improved: For $W$ equal to the hemisphere $\{{x}\in \mS^n: x_{n}\geq 0\}$, and $W^{+} = W\cap\{x_1>0\}$, $W^{-} = W\cap\{x_1<0\}$, the function $v({x}) = x_1$ is a first Dirichlet-Neumann eigenfunction on $W^{\pm}$, with $\mu\left(W^{\pm}\right) = n-1$, leading to equality in  \emph{Theorem \ref{thm:convex}}.  More generally, the inequality
is sharp for any geodesically convex region whose boundary
contains two antipodal points.  After rotation, such regions can be written as
\[
W = \mS^{n-1} \cap (\R\times \Gamma)
\]
with $\Gamma$ a convex cone in $\R^{n-1}$.   The function $v({x}) = x_1$ is again a first Dirichlet-Neumann eigenfunction on $W \cap \{ \pm x_1 >0\}$, and so this also leads to equality in  \emph{Theorem \ref{thm:convex}}. In future work, we show that this is the only case of equality. 
\end{rem}
In this section we prove Theorem \ref{thm:convex} under Assumption \ref{assum:smooth} by obtaining a lower bound on each eigenvalue $\mu\left(W^{\pm}\right)$ individually. The proof will go as follows: We first form a closed manifold by doubling across the convex boundary. By ensuring a lower bound on the Ricci curvature of these closed manifolds, we can then apply the L\'evy-Gromov isoperimetric inequality \cite{Gr}, and Dirichlet eigenvalue lower bounds, \cite{BM}, \cite{FH} to obtain a lower bound on the characteristic exponents. More precisely, we will use the following isoperimetric inequality and eigenvalue estimates:
\begin{thm}[L\'evy-Gromov Isoperimetric Inequality, \cite{Gr} page 2] \label{thm:Gromov}
Let $\mathbb{M}_{\kappa}$ be the $(n-1)$-dimensional space of constant curvature $\kappa>0$, and let $(N,h)$ be a compact, smooth $(n-1)$-dimensional Riemannian manifold with Ricci curvature satisfying $\emph{Ric}_{h}(\xi,\xi) \geq \kappa|\xi|_{h}^2.$ Define the constant $\beta$ by
\begin{align*}
\beta = \frac{\emph{Vol}_{n-1}(N)}{\emph{Vol}_{n-1}(\mathbb{M}_{\kappa})}. 
\end{align*}
If $U\subset N$ is a finite union of domains with smooth boundary, and $\mathbb{D}_{U}$ is a geodesic disc in $\mathbb{M}_{\kappa}$ such that
\begin{align*}
\emph{Vol}_{n-1}(U) = \beta \emph{Vol}_{n-1}(\mathbb{D}_{U}),
\end{align*}
then we have the inequality
\begin{align*}
\emph{Vol}_{n-2}(\pa U) \geq \beta\emph{Vol}_{n-2}(\pa \mathbb{D}_{U}). 
\end{align*}
Here $\emph{Vol}_{n-1}(\cdot)$ and $\emph{Vol}_{n-2}(\cdot)$  measure the $(n-1)$ and $(n-2)$-dimensional volume  on the manifolds respectively. 
\end{thm}
\begin{thm}[B\'erard-Meyer, \cite{BM} 
 Theorem 5] \label{thm:Berard}
Let $\mathbb{M}_{\kappa}$, $(N,h)$, $\beta$, $U$ and $\mathbb{D}_{U}$ be as in Theorem \ref{thm:Gromov}. Then, setting $\mu(U)$, $\mu(\mathbb{D}_{U})$ to be the first Dirichlet eigenvalues of $U$, $\mathbb{D}_{U}$, we have the inequality
\begin{align*}
\mu(U) \geq \mu(\mathbb{D}_{U}).
\end{align*}
\end{thm}

\begin{thm}[Friedland-Hayman, \cite{FH} Theorem 3] \label{thm:FH}
Let $E^{\pm}$ be disjoint subsets of  $\mathbb{S}^{n-1}$. Setting $\mu\left(E^{\pm}\right)$ to be the first Dirichlet eigenvalue of $E^{\pm}$ and $\alpha\left(E^{\pm}\right)$ to be the characteristic exponent,
we have
\begin{align*}
\alpha\left(E^{+}\right) + \alpha\left(E^{-}\right)\geq 2.
\end{align*}
\end{thm}
\begin{rem} \label{rem:FH}
Equality in Theorem \ref{thm:FH} is attained precisely when $E^{\pm}$ are complementary hemispheres, and there is a quantitative positive lower bound on $\alpha\left(E^{+}\right) + \alpha\left(E^{-}\right) - 2$ in terms of the distance of $E^{\pm}$ from hemispheres, \emph{\cite{BKP}}.
\end{rem}
To use these estimates, in the next proposition we construct a double of a smooth, strictly convex subset of $\mS^{n-1}$ to form a closed manifold, while maintaining a lower bound on its Ricci curvature.
\begin{prop} \label{prop:doubling}
Let $W$ be a smooth, closed, strictly convex subset of the sphere 
$\mathbb{S}^{n-1}$, and let $\eta>0$ be given. Then, there exists a smooth metric $g_{\eta}$ on $W$  with the following properties:
\begin{enumerate}

\item The metric $g_{\eta}$ uniformly approximates $g$ $($the round metric on $\mS^{n-1}$$)$, so that for all ${x}\in W$,
\begin{align*}
|g({x}) - g_\eta({x})| < \eta.
\end{align*}

\item The metric $g_{\eta}$ satisfies a symmetry property at the boundary $\pa W$ so that by reflecting the metric $g_{\eta}$ across $\pa W$, we obtain a smooth Riemannian metric on a closed, doubled manifold, denoted by $(M,g_{\eta})$. The boundary of the manifold $(W,g_{\eta})$ is strictly convex.

\item The Ricci curvature of $(M,g_{\eta})$ satisfies the property that $\emph{Ric}_{g_{\eta}} -(1-\eta)g_{\eta}$ is positive definite.

\end{enumerate}
\end{prop}
Before proving the proposition, we first use it to complete the proof of Theorem \ref{thm:convex}. Applying the result of  Proposition \ref{prop:doubling}, we make some definitions:
\begin{defn} \label{defn:doubled-sets}
Let $(M,g_{\eta})$ be the closed, doubled manifold obtained by applying Proposition \ref{prop:doubling} to the set $W$. Via the doubling of $W$, this also gives submanifolds $M^{\pm} \subset M$, with boundaries  ${\Gamma}^{\pm}$ consisting of two glued copies of the $(n-2)$-dimensional submanifolds ${\gamma}^{\pm}$.
\end{defn}
\begin{defn} \label{defn:doubled-eigenvalue}
Let $\mu^{\pm}_{\eta}$ be the first eigenvalue for $(M^{\pm},g_{\eta})$, with Dirichlet boundary conditions on ${\Gamma}^{\pm}$, and let
\begin{align*}
\alpha^{\pm}_{\eta} = -\frac{n-2}{2} + \sqrt{\frac{(n-2)^2}{4} + \mu_{\eta}^{\pm}}
\end{align*}
be the corresponding characteristic exponent.
\end{defn}
We can reduce the proof of Theorem \ref{thm:convex} to studying the eigenvalues $\mu^{\pm}_{\eta}$ via:
\begin{lem} \label{lem:eigenvalue-limit}
The eigenvalues satisfy $\lim_{\eta\to0} \mu^{\pm}_{\eta} = \mu\left(W^{\pm}\right)$.
\end{lem}
\begin{proof}{Lemma \ref{lem:eigenvalue-limit}}
 Let $u^{\pm}_{\eta}\geq0$ be the eigenfunction corresponding to $\mu^{\pm}_{\eta}$, with $L^2({M}^{\pm})$-norm equal to $1$. By the uniqueness of the first eigenfunction, $u^{\pm}_{\eta}$ must be symmetric across the glued boundary. Therefore, denoting $v^{\pm}_{\eta}$ to be the restriction of $u^{\pm}_{\eta}$ to ${W}^{\pm}$, we obtain that for any test function $\psi \in X^{\pm}$ we have
\begin{align} \label{eqn:eta1}
\int_{{W}^{\pm}} \nabla_{g_{\eta}}v_{\eta}^{\pm}(x)\cdot \nabla_{g_{\eta}}\psi(x)\ud \sigma_{g_{\eta}} = \mu^{\pm}_{\eta} \int_{{W}^{\pm}} v^{\pm}_{\eta}({x}) \psi({x})\ud \sigma_{g_{\eta}}.
\end{align}
Here $X^{\pm}$ is given by $X^{\pm} = \{\phi \in C^{\infty}({W}^{\pm}) : \text{supp}(\phi) \cap \gamma^{\pm} =\emptyset , \phi\neq0\}.$  By Proposition \ref{prop:doubling}, the coefficients ${g}^{jk}_{\eta}({x})$ are $C^1$ in $\eta$, and converge to the round metric on the sphere ${g}^{jk}({x})$ as $\eta$ tends to $0$. Moreover, $v^{\pm}_{\eta}$ is a uniformly bounded sequence in $\eta$ in $H^{1}({W}^{\pm})$, and so converges weakly in $H^{1}({W}^{\pm})$ (and strongly in $L^2({W}^{\pm})$) to a function $v^{\pm}$ as $\eta$ tends to $0$. Letting $\eta$ tend to $0$ in \eqref{eqn:eta1}, we thus find that  $\lim_{\eta\to0}\mu^{\pm}_{\eta}$ exists, and we have the equality
\begin{align*}
\int_{{W}^{\pm}} \nabla_{g}v_{\eta}^{\pm}(x)\cdot \nabla_{g}\psi(x)  \ud\sigma_{g}= \left(\lim_{\eta\to0}\mu^{\pm}_{\eta}\right) \int_{{W}^{\pm}} v^{\pm}({x}) \psi({x})\ud \sigma_{g}
\end{align*}
for all $\psi\in X^{\pm}$. Thus, this limit is equal to $\mu\left(W^{\pm}\right)$ as required.  
\end{proof}
Lemma \ref{lem:eigenvalue-limit} allows us to reduce matters to obtain a lower bound on $\mu_{\eta}^{\pm}$, and we now apply the estimates recorded at the start of the proof: We apply Theorems \ref{thm:Gromov} and \ref{thm:Berard} to $\kappa = 1-\eta$, $N = M$, $U = M^{\pm}$. Rescaling Theorem \ref{thm:FH} to the sphere of constant curvature $\kappa = 1-\eta$,  therefore implies that 
\begin{align*}
\alpha_{\eta}^{+} + \alpha_{\eta}^{-} \geq \gamma(\eta),
\end{align*}
with $\lim_{\eta\to0} \gamma(\eta) = 2$. Combining this with Lemma \ref{lem:eigenvalue-limit} completes the proof of  Theorem \ref{thm:convex} under Assumption \ref{assum:smooth}. We therefore are left to prove Proposition \ref{prop:doubling}.

\begin{proof}{Proposition \ref{prop:doubling}}
The definition of the metric $g_{\eta}$ will follow the construction in \cite{AMW} and \cite{W}. Once we have a $C^2$ metric with the three properties in the statement of the theorem it is straightforward to mollify to obtain a smooth metric satisfying the same curvature lower bounds (up to an arbitrarily small error). Therefore, we will construct a $C^2$ metric. 
\\
\\
For each $\eta>0$, we first construct a piecewise smooth, $C^1$ metric $G_{\eta}$ which satisfies the three properties away from a seam at a small distance $\rho$ from $\pa W$. We will then construct a $C^2$ metric $g_{\eta}$, which smooths out this seam while still maintaining the lower bounds on the Ricci curvature. 
\\
\\
\textbf{Step 1: Construction of the piecewise smooth, $C^1$ metric $G_{\eta}$:} We work in local coordinates around a given point ${x}$ on $\pa W$: Let $D$ be a neighbourhood of ${x}$ in $W$ with local coordinates $(r,y)$, where $r$ is the geodesic distance to $\pa W$, $d_{g}(\cdot,\pa W)$, with $r>0$ inside $W$. In these coordinates, the metric $g$ is given by
\begin{align} \label{eqn:localmetric}
g(r,y) = \ud r^2 + h(r,y).
\end{align}
Here $0\leq r < r_0$, for some small fixed constant $r_0>0$, and $h(r,\cdot)$ is a metric on an open submanifold of $\mS^{n-2}$. We then define the slice $\Sigma_r$ by $\Sigma_r = \{(r,y) \in D:d_g((r,y),\pa W) = r \}$, and denote $II^r_g$ to be the second fundamental form of $\Sigma_r$ (computed with respect to the outward, $-\pa_r$, normal). The strict convexity assumption ensures that $II^0_g  > 0$ on $\pa M$.  Thus, by the smoothness of the metric, by taking $r_0>0$ sufficiently small, we may assume that the second fundamental forms $II^r_g$ of $\Sigma_r$ are strictly positive definite. Using the local expression for the metric in \eqref{eqn:localmetric}, this implies that
\begin{align} \label{eqn:localconvexity}
\pa_rh(r,y) <0,
\end{align}
Here $<$ denotes strict negative definiteness. We  now define a metric $G_{\eta}$. Fix a small constant $\rho = \rho(\eta)>0$, with $0<\rho<r_0$, which we will choose below. We then interpolate the metric $h(r,\cdot)$ across the slice $\Sigma_\rho$ with a metric $h_\rho(r,\cdot)$ which has the desired reflection symmetry across $\pa W$:
\begin{defn} \label{defn:g1}
We define the metric $G_{\eta}$ in local coordinates $(r,y)$ on $D$ by
$$
G_{\eta}(r,y) = \ud r^2 + H_{\eta}(r,y) =  \begin{cases}
\ud r^2 + H_\rho(r,y), & \emph{ if } 0\leq r \leq \rho \\
 \ud r^2 + h(r,y) , & \emph{ if } \rho< r <r_0.
\end{cases}
$$
The part of the metric $H_\rho(r,y)$ is defined by
\begin{align*}
{H}_\rho(r,y) = r^2b_{\rho}(y) + c_{\rho}(y),
\end{align*}
for functions $\left(b_{\rho}\right)_{ij}(y)$ and $\left(c_{\rho}\right)_{ij}(y)$ given below.
\end{defn}
Since $H_\rho(r,y)$ has no linear term in $r$, the metric $G_{\eta}(r,y)$ has the desired reflection symmetry across $r=0$. Also, this definition ensures that the slices for fixed $r$ are still equal to $\Sigma_r$  for $0\leq r\leq \rho$. The functions $\left(b_{\rho}\right)_{ij}(y)$ and $\left(c_{\rho}\right)_{ij}(y)$ are chosen to ensure that $H_\rho^r(y)$ and its first derivative in $r$ matches with $h^r(y)$ on $\Sigma_{r}$. Namely, we set
\begin{align*}
\left(b_{\rho}\right)_{ij}(y) = \frac{\pa_rh_{ij}(\rho,y)}{2\rho}, \qquad \left(c_{\rho}\right)_{ij}(y) = h_{ij}(\rho,y) - \rho\frac{\pa_rh_{ij}(\rho,y)}{2}. 
\end{align*}
Thus, the metric $G_{\eta}$ is $C^1$ on $D$, and we may choose $\rho = \rho(\eta)>0$ sufficiently small so that
\begin{align} \label{eqn:g1bounds}
\left|h_{ij}(r,y)-\left(H_\rho\right)_{ij}(r,y)\right|<\eta, \qquad \left|\pa h_{ij}(r,y)-\pa\left(H_\rho\right)_{ij}(r,y)\right|<\eta,
\end{align}
for $0\leq r \leq \rho$, $i,j = 1,\ldots,n-2$, where $\pa$ corresponds to a first derivative in $r$ or $y$. 
\\
\\
For this definition of $G_{\eta}$, with $\rho$ sufficiently small, we claim that the Ricci curvature of $(M,G_{\eta})$ satisfies the property that 
\begin{align} \label{eqn:g1Ricci}
\text{Ric}_{G_{\eta}}-G_{\eta} \text{ is positive definite for } r\neq \rho.
\end{align}
Since we have only modified $g$ for $r\leq \rho$, to show this we can restrict to this range of $r$. By the strict convexity assumption (and taking $\rho>0$ sufficiently small), we can find $\Lambda>0$ so that the eigenvalues of $\pa_rh(\rho,y)$ (computed with respect to the metric $h(r,y)$) are smaller than $- 2\Lambda$. Thus, we also have 
\begin{align} \label{eqn:g1Ricci1}
\pa_r^2{H}_\rho(r,y) + \frac{\Lambda}{\rho}{H}_\rho(r,y) \leq 0
\end{align}
for $0\leq r\leq \rho$. We can now obtain a lower bound on the Ricci curvature by using \eqref{eqn:g1Ricci1}. Denoting the tangential derivatives on $\Sigma_{r}$ by $\pa_{i}$ for $1\leq i \leq n-2$, the sectional curvatures $K_{G_{\eta}}(\cdot,\cdot)$ are given by
\begin{align*}
K_{G_{\eta}}(\pa_i,\pa_j)(r,y) = K_{H_{\rho}(r,\cdot)}(\pa_i,\pa_j)(y) + \frac{1}{4}\left(\pa_r(H_{\rho})_{ij}(r,y)\pa_r(H_{\rho})_{ij}(r,y)-\pa_r(H_{\rho})_{ii}(r,y)\pa_r(H_{\rho})_{jj}(r,y)\right),
\end{align*}
where the first term on the right hand side corresponds to the sectional curvature on $\Sigma_{r}$ with respect to the metric $H_\rho(r,y)$. In particular, these sectional curvature are bounded by an absolute constant (independent of $\rho>0$). Using \eqref{eqn:g1Ricci1}, we also have that $K_{G_{\eta}}(\pa_i,\pa_r)(r,y)$ equals
\begin{align*}
 -\frac{1}{2}\pa_r^2(H_\rho)_{ii}(r,y) &+\frac{1}{4}\left({H}_\rho\right)^{pl}(r,y)\pa_r({H}_\rho)_{ip}(r,y)\pa_r({H}_\rho)_{il}(r,y) \\& \geq \frac{\Lambda}{2\rho}\left({H}_\rho\right)_{ii}(r,y) + \left({H}_\rho\right)^{pl}(r,y)b_{ip}(y)b_{il}(y)r^2  \geq \frac{c_1\Lambda}{\rho},
\end{align*}
for some absolute constant $c_1>0$ (depending on $W$, but independent of $\rho>0$). Thus, by the definition of the Ricci curvature tensor, we see that
\begin{align*}
\text{Ric}_{G_{\eta}}(\pa_r,\pa_r)(r,y)  \geq \frac{c_1\Lambda}{\rho}, \qquad
\text{Ric}_{G_{\eta}}(\pa_i,\pa_i)(r,y)  \geq \frac{c_1\Lambda}{\rho} - C_1.
\end{align*} 
Therefore, by taking $\rho>0$ sufficiently small, we can ensure that \eqref{eqn:g1Ricci} holds. 
\\
\\
\textbf{Step 2: Construction of the $C^2$ metric $g_{\eta}$:} The metric $G_{\eta}$ constructed above fails to be $C^2$ only because it does not have two continuous derivatives in $r$ across the slice $\Sigma_{\rho}$. We therefore will now mollify the metric across $r = \rho$, while maintaining a lower bound on the Ricci curvature. 
\begin{defn} \label{defn:g2}
The metric 
\begin{align*}
g_{\eta}(r,y) = \ud r^2 + h_{\eta}(r,y)
\end{align*}
is given as follows: Fix $\tau>0$ with $\tau<\rho/100$. For $r\in[\rho-\tau,\rho+\tau]$, define $\left(h_{\eta}\right)_{ij}(r,y)$ to be equal to the polynomial of degree $5$ in $r$, which agrees with  $\left(H_{\eta}\right)_{ij}(r,y)$ at $r=\rho\pm\tau$  up to and including the second derivative in $r$. For $r\notin[\rho-\tau,\rho+\tau]$, set $h_{\eta}(r,y)$ to be equal to $H_{\eta}(r,y)$.
\end{defn}
The metric $g_{\eta}$ satisfies the same reflection property as $G_{\eta}$ across $r=0$. By definition, the metric $g_{\eta}(r,x)$ is $C^2$ smooth in $r$ and $x$. Moreover, since $G_{\eta}(r,y)$ is $C^1$-smooth in $r$,  we may choose $\tau = \tau(\eta)>0$ sufficiently small so that 
\begin{align} \label{eqn:g2}
\left|\left(H_{\eta}\right)_{ij}(r,y) - \left(h_{\eta}\right)_{ij}(r,y)\right|<\eta,\qquad  \left|\pa_r\left(H_{\eta}\right)_{ij}(r,y) - \pa_r\left(h_{\eta}\right)_{ij}(r,y)\right|<\eta.
\end{align}
The metric $G_{\eta}(r,y)$ is smooth in $x$ for fixed $r$, and so we can obtain the same inequalities for any number of derivatives in $y$. However, the second derivatives in $r$ of $\left(H_{\eta}\right)_{ij}(r,y)$ do not match for $r = \rho_{-}$ and $r = \rho_{+}$, and so the analogous bound on $\pa^2_r\left(H_{\eta}\right)_{ij}(r,y) - \pa^2_r\left(h_{\eta}\right)_{ij}(r,y)$ will not in general hold. Despite this we will show that there exists a constant $C_1$ (depending only on $W$) such that
\begin{align} \label{eqn:g2Ricci}
\text{Ric}_{g_{\eta}}-\left(1-C_1\eta\right)g_{\eta} \text{ is positive definite}.
\end{align}
This follows provided we can obtain (for $\tau$ sufficiently small) the one-sided bound 
\begin{align} \label{eqn:g2Ricci2}
K_{g_{\eta}}(\pa_i,\pa_r)(r,y) \geq \min\{K_{G_{\eta}}(\pa_i,\pa_r)(\rho-\tau,y),K_{G_{\eta}}(\pa_i,\pa_r)(\rho+\tau,y)\} - C_2\eta.
\end{align}
This is because by the estimates in \eqref{eqn:g2},
\begin{align*}
\left|K_{g_{\eta}}(\pa_i,\pa_j)(r,y) - K_{G_{\eta}}(\pa_i,\pa_j)(r,y)\right|\leq C_3\eta
\end{align*}
for the other sectional curvatures, with the constant $C_3$ depending only on $W$ as usual. From the formula
\begin{align*} 
K_{g_{\eta}}(\pa_i,\pa_r)(r,y)  = -\frac{1}{2}\pa^2_r\left(h_{\eta}\right)_{ii}(r,y) +\frac{1}{4}\left(\left(h_{\eta}\right)(r,y)\right)^{pl}\pa_r\left(h_{\eta}\right)_{ip}(r,y)\pa_r\left(h_{\eta}\right)_{il}(r,y),
\end{align*}
we see that $\pa_r^2 (h_\eta)_{ii}(r,y)$ appears linearly. Therefore, using \eqref{eqn:g2} again,  \eqref{eqn:g2Ricci2} follows from the bound
\begin{align} \label{eqn:g2Ricci1}
\pa^2_r\left(h_{\eta}\right)_{ii}(r,y) \leq \max\left\{ \pa^2_r\left(H_{\eta}\right)_{ii}(\rho-\tau,y),\pa^2_r\left(H_{\eta}\right)_{ii}(\rho+\tau,y)\right\} + C_4\eta,
\end{align}
for all $\rho-\tau \leq r \leq \rho+ \tau$. 
\\
\\
To complete the proof it remains to show that \eqref{eqn:g2Ricci1} holds for $\tau = \tau(\eta)>0$ sufficiently small. We have $\pa^2_r\left(H_{\eta}\right)_{ii}(\rho-\tau,y)\leq \pa^2_r\left(H_{\eta}\right)_{ii}(\rho+\tau,y)$. Therefore, for each $y$ and $i$ fixed, we define the quintic polynomial  $F(t) = (H_{\eta})_{ii}(\rho-\tau + t,y)$
and it is sufficient to show that
\begin{align*}
F''(t) \leq F''(1) +C_4\eta
\end{align*}
for all $t$, $0\leq t \leq 2\tau$. By subtracting a quadratic polynomial we may assume that $F(0) = F'(0) = F''(0) = 0$, and then by \eqref{eqn:g2} we have $|F(2\tau)| + |F'(2\tau)| \leq C\tau$, together with $F''(1)\geq0$. The upper bound on $F''(t)$ then follows for $\tau$ sufficiently small by directly computing the second derivative.
\end{proof}

\section{A Dini condition for the characteristic exponents} \label{sec:FH2}

In this section, we will use the lower bound on the characteristic exponents from Theorem \ref{thm:convex} to prove Proposition \ref{prop:Main2}.  
Since the sets $V_{t}\subset\mathbb{S}^{n-1}$ are not necessarily convex, we cannot apply Theorem \ref{thm:convex} directly.  Instead, 
we use the fact that Assumption \ref{assum:Main}, the Dini condition, holds.  
The key step is to approximate $V_{t}$ by a convex subset of $\mS^{n-1}$.
\begin{prop} \label{prop:convex-approx}
For each $t\in(0,1)$, there exists a convex set $W_{t}\subset V_{t}$ such that denoting $A(t)$ by $A(t) = \emph{Haus}_{\mS^{n-1}}\left(V_{t}, W_{t}\right)$, the bound
\begin{align*}
A(t) \leq C_1t^{-1}M_{\etb}(4t)
\end{align*}
holds.  Here $C_1$ is a constant depending only on the Lipschitz constant of $\pa\Omega$.
\end{prop}
\begin{rem} \label{rem:convex-approx}
In fact the control on the volume $\emph{Vol}_{\mS^{n-1}}\left(V_{t}\backslash W_{t}\right)$, together with a modest regularity statement concerning $\pa V_t$ will be sufficient to combine this proposition with Theorem \ref{thm:convex} in order to obtain a lower bound on the characteristic constants and prove Proposition \ref{prop:Main2}. 
\end{rem}
We will show how this proposition implies Proposition \ref{prop:Main2}
and then prove Proposition \ref{prop:convex-approx} at the end of this section.

For $A(t)$ as in Proposition \ref{prop:convex-approx} we combine the eigenvalue lower bounds from Theorem \ref{thm:convex} to prove:
\begin{prop} \label{prop:eigenvalue-approx}
There exists a constant $C_2>0$, depending on the Lipschitz constant of $\pa\Omega$ such that for each $t\in(0,1)$, we have the lower bound on the characteristic constants from Definition \ref{defn:DN} given by
\begin{align*} 
\left[2-\al^{+}(t) - \al^{-}(t)\right]_{+} \leq C_2t^{-1}M_{\etb}(4t).
\end{align*}
\end{prop}
Combining Proposition \ref{prop:eigenvalue-approx} with the Dini condition on $t^{-1}M_{\textbf{0}}(t)$ from Assumption \ref{assum:Main} implies Proposition \ref{prop:Main2}. We are therefore left to prove Propositions \ref{prop:convex-approx} and \ref{prop:eigenvalue-approx}. To prove Proposition \ref{prop:eigenvalue-approx}, we will use the following lemma to convert estimates on the volume between $V_t$ and $W_t$ to the corresponding Dirichlet-Neumann eigenvalues. 
\begin{lem} \label{lem:eigenvalue-lower}
Let $W \subset V$ be subsets of a smooth manifold $(M,g)$, and suppose that the boundaries of $W$ and $V$ are of the following form: There exist non-empty sets $\pa V_{N}$, $\pa V_{D}$, $\pa W_{N}$, and $\pa W_{D}$ such that
\begin{align*}
\pa V = \pa V_{N}\cup \pa V_{D}, \qquad \pa W = \pa W_{N}\cup \pa W_{D},
\end{align*} 
$\pa W_{D} \subset \pa V_{D}$, and $\pa V_{N}$ is Lipschitz. Then, setting $\mu(V)$ to be the first eigenvalue of $V$, with Neumann boundary conditions on $\pa V_N$, and Dirichlet boundary conditions on $\pa V_D$, and $\mu(W)$ likewise, we have
\begin{align*}
\mu(W) \leq \mu(V) \left(1+ C\emph{Vol}_{M}(V\backslash W)\right).
\end{align*}
Here $C$ depends only on $\mu(V)$ and the Lipschitz constant of $\pa V_{N}$, and $\emph{Vol}_{M}(\cdot)$ denotes area on $M$.
\end{lem}
\begin{proof}{Lemma \ref{lem:eigenvalue-lower}}
Let $w_{V}$ be a $L^2(V)$-normalised eigenfunction of $V$, with eigenvalue $\mu(V)$. We claim that there exists a constant $C$,  depending only on $\mu(V)$ and the Lipschitz constant of $\pa V_{N}$ such that
\begin{align} \label{eqn:eigenvalue-lower1}
\norm{w_{V}}_{L^{\infty}(V)} \leq C.
\end{align}
Assuming that \eqref{eqn:eigenvalue-lower1} holds, we can now obtain an upper bound on $\mu(W)$: Let $w_{W}$ be the restriction of $w_{V}$ to $W$. Then, by the assumption on the sets $W$ and $V$, $w_{W}$ is an admissible test function, and so
\begin{align*}
\mu(W) \leq \frac{\displaystyle \int_{W} \left|\nabla_{g}w_{W}\right|^2\ud \sigma}{\displaystyle \int_{W} \left|w_{W}\right|^2\ud \sigma} \leq  \frac{\displaystyle \int_{V} \left|\nabla_{g}w_{V}\right|^2\ud \sigma}{\displaystyle \int_{W} \left|w_{W}\right|^2\ud \sigma}  = \frac{\mu(V)}{\displaystyle \int_{W} \left|w_{W}\right|^2\ud \sigma}.
\end{align*}
Since
\begin{align*}
 \int_{W} \left|w_{W}\right|^2\ud \sigma  = 1 -  \int_{V\backslash W} \left|w_{V}\right|^2\ud \sigma \geq 1-C^2\text{Vol}_{M}(V\backslash W),
\end{align*}
this gives the desired bound, provided \eqref{eqn:eigenvalue-lower1} holds. To obtain \eqref{eqn:eigenvalue-lower1}, let $\tilde{v}_V$ be the reflection of $v_V$ across the Lipschitz boundary $\pa V_N$. Since $v_{V}$  satisfies Neumann boundary conditions on $\pa V_N$  the function $\tilde{v}_V$ is the solution of an elliptic equation in divergence form, with bounded measurable coefficients on the doubled domain of $V$. Moreover, $\tilde{v}_V$ vanishes on the boundary of this domain, and the coefficients of the elliptic operator can be bounded in terms of the Lipschitz constant of $\pa V_N$, and $\mu(V)$. Therefore, Theorem 8.25 in \cite{GT} implies the desired $L^{\infty}$ control on $\tilde{v}_V$ and hence $v_V$ itself.
\end{proof}
\begin{proof}{Proposition \ref{prop:eigenvalue-approx}}
In order to apply Lemma \ref{lem:eigenvalue-lower} to $V_t$ and $W_t$, we need to establish the Lipschitz regularity of $\pa V_t$. Recall from Definition \ref{defn:Sc}, the sets $S_c = \{t\in(0,1):s^{-1}M_{\textbf{0}}(s)\leq c \text{ for all } s\in[\tfrac{1}{4}t,4t]\}$. Since $M_{\textbf{0}}(s)$ is increasing in $s$, if $t\notin S_c$, then  by taking $C_2$ sufficiently large, depending on $c$, the estimate in Proposition \ref{prop:eigenvalue-approx} follows immediately. On the other hand, for $t\in S_c$, with $c$ sufficiently small, the sets $V_t$ are more regular.  
 \begin{lem} \label{lem:graph}
There exists a constant $c>0$, depending only on the Lipschitz norm of $\pa\Omega$, and an orientation of $\Omega$ with the following property: For $t\in S_{c}$, the part of the boundary $\pa\Omega \cap \left(B_{2t}\backslash B_{t/2}\right)$ can be written as the graph $x_n = g(x')$ of a convex function $g$, with Lipschitz constant depending only on that of $\pa\Omega$. Here we have written $x = (x',x_n)\in\R^n$.
 \end{lem}
 \begin{cor} \label{cor:graph}
For $c$ and the orientation of $\Omega$ as in Lemma \ref{lem:graph}, and for $s\in[\tfrac{1}{2}t,2t]$ with $t\in S_c$, the slices $V_s$ are star-shaped with respect to the north pole in $\mathbb{S}^{n-1}$. Moreover, there exists a neighbourhood around each $y\in\pa V_s$ on which $\pa V_s$ can be parameterized as submanifold $(x'(\tau),g(x'(\tau)))$ for $\tau$ in an open set in $\R^{n-2}$. The Lipschitz constant of this parameterization and the size of the neighbourhood can be taken to only depend on that of $\pa\Omega$.
\end{cor}
We will prove Lemma \ref{lem:graph} and Corollary \ref{cor:graph} in the Appendix.
 \begin{rem} \label{rem:graph}
We will in fact show in the proof of Corollary \ref{cor:graph} that after a rotation so that $y$ lies in the $(x_1x_n)$-plane, the Lipschitz parameterization of $\pa V_s$ can be written as $(x_1(\tau),\tau,s^{-1}g(sx_1(\tau),s\tau)$, where $\tau$ is contained in an open neighbourhood of the origin in $\R^{n-2}$.
\end{rem}
Now fix $t\in S_c$, with $c$ as in Lemma \ref{lem:graph}. Let $W_t$ be as in Proposition \ref{prop:convex-approx} and let $U_t^{\pm}\subset V_t$ and $\gamma_{t}^{\pm}$ be as in \eqref{eqn:Ui}. Suppose first that one of $\gamma_{t}^{\pm}$ does not intersect $W_{t}$. Then, one of $U_t^{\pm}$  is contained within $V_{t}\backslash W_{t}$. Therefore, there exists a constant $c_0>0$ such that if $A(t) = $ Vol$_{\mathbb{S}^{n-1}}(V_t\backslash W_t) <c_0$, then  we have $\al^{+}(t) \geq 2$ or $\al^{-}(t) \geq 2$, and the estimate follows immediately. But if $A(t)\geq c_0$, then by Proposition \ref{prop:convex-approx}, the estimate in the proposition is trivial by taking $C_2$ sufficiently large.
\\
\\
Therefore, we assume that at least one of $\gamma_{t}^{\pm}$ intersects $W_t$. We write
\begin{align*}
W_t^{+} = \{u_t>0\}\cap W_t, \quad W_t^{-} = \{u_t<0\}\cap W_t.
\end{align*}
Setting $\mu^{\pm}(t)$ to be the first Dirichlet-Neumann eigenvalue on $W_t^{\pm}$, with Dirichlet boundary conditions on $\gamma_t^{\pm}\cap \bar{W}_t^{\pm}$, and Neumann boundary conditions on the rest of $\pa W_t^{\pm}$, we have the following two estimates: Since $W_{t}\subset \mS^{n-1}$ is geodesically convex, by Theorem \ref{thm:convex} (or the classical Friedland-Hayman inequality \cite{FH}, when $W_t = \mS^{n-1}$), we have
\begin{align} \label{eqn:eigenvalue-approx1}
\left(-\frac{n-2}{2}+ \sqrt{\frac{(n-2)^2}{4} + \mu^{+}(t)}\right)+\left( - \frac{n-2}{2} + \sqrt{\frac{(n-2)^2}{4}+ \mu^{-}(t)}\right)  \geq 2.
\end{align}
By Corollary \ref{cor:graph}, $\pa V_t$ is Lipschitz and so we can apply Lemma \ref{lem:eigenvalue-lower}, we have 
\begin{align} \label{eqn:eigenvalue-approx2}
\mu^{\pm}(t) \leq \la^{\pm}(t)\left(1+C A(t)\right),
\end{align}
and combining \eqref{eqn:eigenvalue-approx1} and \eqref{eqn:eigenvalue-approx2} with Proposition \ref{prop:convex-approx} gives the desired estimate. 
\end{proof}

We are left to prove Proposition \ref{prop:convex-approx}. 
\begin{proof}{Proposition \ref{prop:convex-approx}}
We first note that by taking $C_1$ sufficiently large, we can restrict to those values of $t$ for which Lemma \ref{lem:graph} and Corollary \ref{cor:graph} apply. We first obtain a lower bound on the principal geodesic curvatures of $\pa V_t$ at a point $y = (y',y_n)$, and then construct a nearby set with positive principal geodesic curvatures. After rotating in the $x'$ variables, we may assume that $y = (y_1,0,\ldots,0,y_n)$, with $y_1>0$. For the first step, we take a direction in the $x_2,\ldots,x_{n-1}$-variables, which after a rotation we take to be the $x_2$-direction, and intersect $\pa V_t$ with the hyperplane $x_3=\cdots =x_{n-1}= 0$. This leads to a curve, and we will find a lower bound on its geodesic curvature. Using Lemma \ref{lem:graph}, this intersection is equal to $(x_1,x_2,g_t(x_1,x_2))$, where we have written $g_t(x_1,x_2) = t^{-1}g(tx_1,tx_2)$, and to ease notation we have omitted writing out $0$ for the remaining coordinates. By Corollary \ref{cor:graph} and Remark \ref{rem:graph}, near the point $y$ this intersection is a curve $\alpha(\zeta)$ for $\zeta$ contained in an interval around $0$, with $\alpha(0) = (y_1,0,g_t(y_1,0))$ and 
\begin{align*}
\alpha(\zeta) = (x_1(\zeta),\zeta,g_t(x_1(\zeta),0)). 
\end{align*}
Here $x_1(\zeta)$ is a Lipschitz function. Since $\alpha(\zeta)$ lies on $\mS^2$, the geodesic curvature of the curve at $\alpha(\zeta)$  equals
\begin{align} \label{eqn:geod1}
\kappa(\zeta) = \frac{\alpha(\zeta) \times\alpha'(\zeta) \cdot \alpha''(\zeta)}{|\alpha'(\zeta)|^3}.
\end{align}
Evaluating \eqref{eqn:geod1} at $\zeta=0$, we have 
\begin{align} \label{eqn:geod2}
\kappa(0) = \frac{-x_1''(0)(g_t -y_1\pa_{x_1}g_t) +y_1(\pa_{x_2}^2g_t  + 2x_1'(0)\pa_{x_1}\pa_{x_2}g_t  + (x_1'(0))^2\pa_{x_1}^2g_t )}{\left(1+(x_1'(0)\pa_{x_1}g_t+\pa_{x_2}g_t  )^2 \right)^{3/2}},
\end{align}
where $g_t$ and its derivatives are evaluated at $(y_1,0)$. Moreover, since $\alpha(\zeta)$ lies on $\mathbb{S}^{2}$, we have $x_1(\zeta)^2+\zeta^2+g_t\left(x_1(\zeta),\zeta\right)^2 = 1$, and so differentiating this twice to eliminate $x_1''(0)$, we can rewrite the expression in \eqref{eqn:geod2} as
\begin{align} \nonumber
& \frac{ \left(\pa_{x_2}^2g_t+2x_1'(0)\pa_{x_1}\pa_{x_2}g_t + (x_1'(0))^2\pa_{x_1}^2g_t\right)\left(g_t^2 +(y_{1})^2\right)}{(y_1+g_t\pa_{x_1}g_t)\left(1+(x_1'(0)\pa_{x_1}g_t + \pa_{x_2}g_t )^2 \right)^{3/2}} \\ \label{eqn:geod3}
& +\frac{\left(1+(x_1'(0))^2 +( x_1'(0)\pa_{x_1}g_t+ \pa_{x_2}g_t )^2\right)(g_t - y_1\pa_{x_1}g_t) }{(y_1+ g_t\pa_{x_1}g_t)\left(1+(x_1'(0)\pa_{x_1}g_t + \pa_{x_2}g_t )^2 \right)^{3/2}}  . 
\end{align}
The denominator in both terms is positive, as is the numerator of the first fraction (by the Euclidean convexity of $g_t$). The quantity $y_1+ g_t\pa_{x_1}g_t$ in the denominator of the second term can be bounded from below for $t\in S_c$, with $c>0$ sufficiently small. Combining this with the Lipschitz bounds on $x_1(\zeta)$ and $g_t$, means that we have a lower bound on the second fraction of 
\begin{align*}
\left(1+(x_1'(0))^2 +(x_1'(0)\pa_{x_1}g_t+ \pa_{x_2}g_t )^2\right)(g_t - y_1\pa_{x_1}g_t)  \geq - C\left(y_1\pa_{x_1}g_t - g_t\right).
\end{align*}
Now consider the set $V_{t,\delta}$, which we obtain from $V_t$ via a vertical shift of $\delta$ in the $x_n$-direction, for a small constant $\delta$. That is, $(x',x_n)\in\pa V_{t,\delta}$ precisely when $x_n = g_{t,\delta}(x') = g_t(x')+\delta$. Just as for $V_t$ above, we obtain a curve $\alpha_{\delta}(\zeta)$ for part of the boundary of $V_{t,\delta}$. Let $x_1 = y_{1,\delta}>0$ be the value where the projection of $\pa V_{t,\delta}$ onto the $x'$-variables passes through the $x_1$-axis. Then,
\begin{align*}
y_{1,\delta}\pa_{x_1}g_{t,\delta}(y_{1,\delta},0) - g_{t,\delta}(y_{1,\delta},0) = y_{1,\delta}\pa_{x_1}g_{t}(y_{1,\delta},0) - g_{t}(y_{1,\delta},0) -\delta.
\end{align*}
In particular, if $\delta > y_{1,\delta}\pa_{x_1}g_{t}(y_{1,\delta},0) - g_{t}(y_{1,\delta},0)$, then we have
\begin{align*}
y_{1,\delta}\pa_{x_1}g_{t,\delta}(y_{1,\delta},0) - g_{t,\delta}(y_{1,\delta},0) < 0,
\end{align*}
and the geodesic curvature of $\alpha_{\delta}$ is positive. Therefore, letting $r$ be the radial coordinate in the $x'$-variables, and setting
\begin{align} \label{eqn:geod4}
\delta(t) = \sup_{|x'|\leq 1}r\pa_rg_t(x')-g_t(x'),
\end{align}
we can form a geodesically convex set $V_{t,\delta}$ such that
\begin{align*}
\text{Haus}_{\mS^{n-1}}\left(\pa V_t, \pa V_{t,\delta(t)}\right) \leq \tilde{C}\delta(t).
\end{align*}
Setting $W_t$ to be $V_{t,\delta(t)}$, to complete the proof of the proposition we need to bound the value of $\delta(t)$ from \eqref{eqn:geod4}: Since
\begin{align*}
r\pa_rg_t(x') - g_t(x') & = \left(\nabla g_t(x'),-1\right)\cdot \left(x',g_t(x')\right), \\
& = \left(1+\left|\nabla g_t(x')\right|^2\right)^{1/2} \nu(tx',tg_t(x')) \cdot (x',g_t(x')) ,
\end{align*}
 there exists a  constant $C$ such that 
\begin{align*}
\delta(t) \leq Ct^{-1}\sup_{{x}\in\pa\Omega, |x| \leq t} \nu({x})\cdot {x} \leq Ct^{-1}M_{\textbf{0}}(t),
\end{align*}
as required. 
\end{proof}

\section{Estimates on a Dirichlet-Neumann harmonic measure and proof of Propositions \ref{prop:Main1} and \ref{prop:Main3}} \label{sec:harmonic}

In this section we prove Propositions \ref{prop:Main1} and \ref{prop:Main3} under the Dini condition, Assumption \ref{assum:Main}. To do this, we first need to prove
a reverse H\"older estimate for harmonic measure under
mixed Dirichlet/Neumann boundary conditions.   

Let $U$ be a convex domain with $\textbf{0}\in \bar U$, and set 
\[
U_1 = B_1(\textbf{0}) \cap U, 
\quad D = \pa B_1(\textbf{0})\cap \bar{U}, 
\quad N = B_1(\textbf{0})\cap\pa U.
\]
For $z\in U_1$, define $G_1^z$ to be Green's function for $U_1$ with
pole at $z$ and Neumann conditions on $N$ and Dirichlet conditions on $D$.

We will impose a transversality condition between the
hypersurfaces $N$ and $D$ as follows.  There is a constant
$c^*>0$ such that for every $0< r < r_1$
and every $p \in \pa B_1(\textbf{0}) \cap \pa U$, there is a unit vector $\tau_1$
such that 
\begin{align} \label{eqn:harmonic1}
\tau_1 \cdot \nu({x}) \leq -c^* \text{ for } x\in D\cap B_{2r}(p), \qquad \tau_1\cdot \nu({x}) \geq c^* \text{ for } x\in N\cap B_{2r}(p),
\end{align}
with $\nu(x)$ the outer unit normal to $U_1$.   In other words, if $\tau_1$
is the downwards vertical direction, the Dirichlet portion is the graph above $U_1$, 
the Neumann portion is the graph below $U_1$, and they
meet at a strictly positive angle.   Since the normal $\nu(x)$ need not be 
unique on the convex boundary $N$, we require the condition \eqref{eqn:harmonic1}
for the unit normal to every support plane.

\begin{prop} \label{prop:harmonic}
With the notations above, let $G_1 = G_1^z$ be Green's function
for $U_1$ with pole $z \in U \cap B_{1/2}(\etb)$.  Set $K_1 = \pa_\nu G_1$,
the density of harmonic measure with respect to surface measure $d\sigma$ 
on the Dirichlet portion of the boundary, $D$ and define the surface
ball
\begin{align*}
S_r(p) = B_r(p) \cap D.
\end{align*}
Then for every $p\in D$, $0 < r < r_1$, 
\begin{align} \label{eqn:harmonic2}
\left(\frac{1}{r^{n-1}}\int_{S_{r}(p)}K_1^2 \ud \sigma \right)^{1/2} \leq C\frac{1}{r^{n-1}} \int_{S_{r}(p)}K_1 \ud \sigma ,
\end{align}
for a constant $C$ depending only on $c^*$ and the Lipschitz constant of $U$. 
\end{prop}
\begin{proof}{Proposition \ref{prop:harmonic}}
Let $\chi({x})\geq0$ be a smooth cut-off function, which equals $1$ in $B_{r}(p)$ and $0$ outside $B_{2r}(p)$. For $v_1 = \chi G_1$, and $\tau_1$ as in \eqref{eqn:harmonic1}, we follow a calculation used by Brown (see Lemma 1.7 in \cite{Br}) in order to gain estimates on $\nabla G_1$ exploiting \eqref{eqn:harmonic1}: We first have the Rellich identity
\begin{align*}
\int_{\pa U_1} \left|\nabla v_1\right|^2 \tau_1\cdot\nu - 2 \pa_{\nu}v_1 \tau_1\cdot\nabla v_1 = -2\int_{U_1} \left(\Delta v_1\right) \tau_1\cdot\nabla v_1. 
\end{align*}
Using \eqref{eqn:harmonic1}, together with $v_1\equiv 0$ on $D$,  this identity implies the  bound
\begin{align} \label{eqn:L2harmonic2}
\int_{D} \left(\pa_\nu v_1\right)^2 + \int_{N}\left|\nabla_{\text{tang}}v_1\right|^2 \leq C_1\int_{U_1} \Delta v_1 \tau_1\cdot\nabla v_1 + C_1\int_{N}\left|\pa_{\nu}v_1\right|^2,
\end{align}
where $\nabla_{\text{tang}}$ is the tangential gradient on $N$, and $C_1$ depends only on $c^*$. We can rewrite the right hand side of \eqref{eqn:L2harmonic2} as
\begin{align} \label{eqn:L2harmonic3}
C_1\int_{U_1} \left(G_1\Delta \chi + 2\nabla G_1\cdot\nabla \chi\right)\tau_1\cdot\nabla\left(\chi G_1\right) + C_1\int_{N}G_1^2\left|\pa_{\nu}\chi \right|^2,
\end{align}
and we have $|\nabla \chi({x})| \leq Cr^{-1}$, $|\Delta \chi({x})| \leq Cr^{-2}$, for a universal constant $C$. Let $\tilde{\chi}$ be another smooth cutoff function, equal to $1$ on the support of $\chi$ and vanishing outside $B_{4r}(p)$. We have
\begin{align*}
\int_{U_1}\tilde{\chi}^2\left|\nabla G_1\right|^2 = -\int_{U_1} \tilde{\chi}\nabla \tilde{\chi}
 \cdot \nabla \left(G_1^2\right),
\end{align*}
since $G_1\pa_{\nu}G_1=0$ on $\pa U_1$, and $G_1$ is harmonic on the support of $\tilde{\chi}$. Integrating by parts again gives
\begin{align} \label{eqn:L2harmonic4}
\int_{U_1}\tilde{\chi}^2\left|\nabla G_1\right|^2 = \int_{U_1} G_1^2\text{div}\left(\tilde{\chi}\nabla \tilde{\chi}\right) + \int_{\pa U_1} G_1^2 \tilde{\chi}\nu\cdot\nabla \tilde{\chi}. 
\end{align}
Since $G_1$ satisfies Neumann boundary conditions on $N$, we can flatten this part of the boundary and reflect across it. Denoting $\tilde{G}_1$ to be the reflection of $G_1$ across $N$, it satisfies an elliptic equation in divergence form, with bounded measurable coefficients, depending only on the Lipschitz constant of $N$. Moreover, $\tilde{G}_1\geq0$ vanishes on the boundary of this doubled domain. Therefore, the boundary Harnack inequality, \cite{CFMS}, applies, giving
\begin{align} \label{eqn:L2harmonic5}
G_1({x}) \leq G_1(q), \text{ for } {x} \in \text{supp}(\tilde{\chi}),
\end{align}
where $q$ is a point in the support of $\tilde{\chi}$ at a distance of $r$ from $\pa U_1$. Using the Litmann, Stampacchia, Weinberger bounds, \cite{LSW}, we can bound $G_1(q)$ by a constant depending only on the Lipschitz constant of $U$. Inserting the estimates from \eqref{eqn:L2harmonic4} and \eqref{eqn:L2harmonic5} into \eqref{eqn:L2harmonic3}  thus gives
\begin{align*}
\int_{D} \left(\pa_\nu v_1\right)^2 + \int_{N}\left|\nabla_{\text{tang}}v_1\right|^2 \leq C_2G_1(q)^2,
\end{align*}
for a constant $C_2$ depending only on $c^*$ and the Lipschitz constant of $U$. Doubling estimates for the elliptic measure, for operators in divergence form with bounded measurable coefficients, \cite{CFMS}, ensure that
\begin{align*}
G_1(q) \leq C\frac{1}{r^{n-1}} \int_{S_{r}(p)}K_1 \ud \sigma,
\end{align*}
and hence the reverse H\"older inequality for $K_1$ in \eqref{eqn:harmonic2} holds. 
\end{proof}
We will use the estimate on $K_1$ in Proposition \ref{prop:harmonic} to prove the following:
\begin{prop} \label{prop:harmonic2}
Let $U$, $U_1$ be as in \emph{Proposition \ref{prop:harmonic}}, and suppose that there exists $r_0>0$ such that for each $p\in D =  \pa B_1(\etb)\cap U$, the inequalities in \eqref{eqn:harmonic1} hold in $U_1\cap B_{2r_0}(p)$, for a constant $c^*>0$. Let $w$ satisfy, in the appropriate non-tangential sense,
\begin{align*}
\Delta w = 0 \text{ in } U_1, \qquad w = f \text{ on } D, 
\qquad \pa_{\nu}w = 0 \text{ on } N = \pa U \cap B_1(\etb)
\end{align*}
for some $f\in L^2(D, d\sigma)$. Then,
\begin{align*}
\sup_{B_{1/2}(\emph{\textbf{0}})\cap U}  w^2 
\leq 
C\int_{D} f^2  \ud \sigma, 
\qquad \int_{N} w^2 \ud \sigma
\leq C\int_{D} f^2\ud \sigma,
\end{align*}
for a constant $C$ depending only on $r_0$, $c^*$, and the Lipschitz constant of $U$. 
\end{prop}
\begin{proof}{Proposition \ref{prop:harmonic2}}  To prove the interior supremum
estimate, fix $z\in B_{1/2}(\textbf{0})\cap U$ and set
$\ud \omega_1 = K_1 \ud \sigma = \pa_\nu G_1^z \ud \sigma$.
Then
\[
w(z) = \int_{D}  f \ud \omega_1 \leq 
\left(\int_{D} f^2\ud \sigma\right)^{1/2} 
\left(\int_{D} K_1^2 \ud \sigma\right)^{1/2} \, .
\]
Covering $D$ with a finite number of balls of fixed radius $r_0$, it
follows from Proposition \ref{prop:harmonic}
that 
\[
\left(\int_{D} K_1^2 \ud \sigma\right)^{1/2}  \leq C\int_D K_1 \ud \sigma = C
\]
with a constant $C$ depending on $r_0$, $c^*$ and the Lipschitz constant
of $U$. 

The second claim, the $L^2(N)$ bound all the way up to the 
places where $N$ meets $D$, uses nontangential maximal function bounds.
Since we have already taken care of interior points, it suffices to consider 
points $q \in N$ such that dist$(q, \pa B_1(\textbf{0})) \le r_0$.
For such points, the transversality assumption \eqref{eqn:harmonic1}
implies that there is a unique point $p\in D$ ``above" $q$ in
the sense that 
\[
q = p + s \tau_1 \quad \mbox{for some } \ s>0.
\]
Since $q$ belongs to the nontangential cone below $p$, 
$w(q) \le N^*(w)(p)$ where $N^*$ is defined as the nontangential maximal function.
Next, recall that the estimates of Hunt and Wheeden (\cite{CFMS}, \cite{HW})
imply that the nontangential maximal function is dominated by the 
Hardy-Littlewood maximal function associated with harmonic measure
$d \omega^* = \pa_\nu G_1^{z^*}$ at a single, fixed central point $z^*$
inside $U_1$:
\begin{align*} 
w(q) \le N^*(w)(p) \leq C_1M^*(f)(p),
\end{align*}
for a constant $C_1$ depending only on the Lipschitz constant of $U$, 
with 
\begin{align} \label{eqn:L2estimate1}
M^*(f)(p) = \sup_{p\in \Delta}\frac{1}{\omega^*(\Delta)} \int_{\Delta}f \ud \omega^*,
\end{align}
for $\Delta$ a surface ball on $D$.  In particular,
\[
\|w\|_{L^2(N,\ud \sigma)} \le C \| M^*(f)\|_{L^2(D, \ud \sigma)}.
\]
Finally, the bound in \eqref{eqn:harmonic2} ensures that $d \sigma$ is in the Muckenhoupt  weight class $A_{2}(d \omega^*)$.  Thus
\begin{align} \label{eqn:L2estimate2}
\norm{M^*(f)}_{L^2(D,\ud \sigma) } 
\leq C_2\norm{f}_{L^2(D, \ud \sigma)}
\end{align}
for a constant $C_2$ depending only on the constant $C$ from Proposition \ref{prop:harmonic}.  Combining \eqref{eqn:L2estimate2} with \eqref{eqn:L2estimate1}
we obtain the second estimate, concluding the proof of the proposition.
\end{proof}

Let the functions $v^{\pm}$ be given by
\begin{align*}
\Delta v^{\pm} = 0 \text{ in }B_t(\textbf{0})\cap \Omega, \qquad v^{\pm} = u^{\pm} \text{ on }\pa B_t(\textbf{0}) \cap\Omega, \qquad \pa_{\nu}v^{\pm} = 0 \text{ weakly on }  B_t(\textbf{0})\cap \pa\Omega. 
\end{align*}
Since $u^{\pm}$ is harmonic in $\{u^{\pm}>0\}$, and $\pa_{\nu}u^{\pm}=0$ weakly on $\{u^{\pm}>0\}\cap B_t(\textbf{0})\cap\pa\Omega$, the functions $v^{\pm}$ maximize $u^{\pm}$ in $B_t(\textbf{0})\cap\Omega$. 
\\
\\
We will apply Propositions \ref{prop:harmonic} and \ref{prop:harmonic2} to the sets $V_t,$ and the functions $v^{\pm}_t$, (which are the rescalings of $v^{\pm}$ by $t$ around $\textbf{0}$). Note that the assumption on the angles in \eqref{eqn:harmonic1} from Proposition \ref{prop:harmonic} can only fail when $\pa B_t(\textbf{0})$ is close to tangential to $\pa\Omega$. We  show that the assumption in \eqref{eqn:harmonic1} holds when $t^{-1}M_{\textbf{0}}(t)$ is sufficiently small.

\begin{lem} \label{lem:transverse}
Define the set $U_t = (t^{-1}\Omega)\cap B_1)$. This has boundary in two parts, given by $V_t$, with outward pointing unit normal $x$, and $\pa\Omega_t = t^{-1}(\pa\Omega \cap B_t)$, with outward pointing unit normal $\nu(x)$. There exist constants $c$ and $r>0$ with the following property: For $t\in S_{c}$, and each $p\in\pa V_t$, there exists a unit direction $\tau_1$ such that
\begin{align*}
\tau_1\cdot x \leq -\tfrac{1}{10} \text{ for } x\in V_t\cap B_{r}(p), \qquad \tau_1\cdot \nu(x) \geq \tfrac{1}{10} \text{ for } x\in \pa\Omega_t\cap B_{r}(p).
\end{align*}
\end{lem}
In particular, this ensures that the condition in \eqref{eqn:harmonic1} holds for this range of $t$, and so we can apply Proposition \ref{prop:harmonic2} to the sets $U = t^{-1}\Omega$ and $U_1 = U_t$. 
\begin{proof}{Lemma \ref{lem:transverse}}
After a rotation in the $x'$-variables, and by taking $c$ sufficiently small so that Lemma \ref{lem:graph} applies, we can write the point $p = (p',p_n)$ as $(p_1,0,\ldots,0,g_t(p'))$, with $p_1>0$. Moreover, by taking $c >0$ sufficiently small, we can ensure that $\nu(p)\cdot p \leq \tfrac{1}{10}$. In particular, we can choose a unit direction $\tau_1$ in the $(x_1,x_n)$-plane such that
\begin{align*}
\tau_1 \cdot p \leq - \tfrac{1}{5}, \qquad \tau_1 \cdot \nu(p)\geq \tfrac{1}{5}.
\end{align*}
For $x\in V_t$, with $|x-p|\leq \tfrac{1}{10}$, we therefore have $\tau_1\cdot x \leq -\tfrac{1}{10}$. For $x\in \pa\Omega_t$, with $|x-p|\leq \tfrac{1}{10}$, we also have that $\nu(x)\cdot x \leq \tfrac{1}{10}$, and so  we can ensure that $\left|\tau_1 \cdot (\nu(p)-\nu(x))\right| \leq \tfrac{1}{10}$ for $|x-p|$ sufficiently small. Therefore, the inequalities in the statement of the proposition hold for $|x-p|<r$ with $r$ sufficiently small.
\end{proof}

We can now prove Propositions \ref{prop:Main1} and \ref{prop:Main3} by estimating the functions $A^{\pm}(t)$ from Definition \ref{defn:A-plus}. The proof of Proposition \ref{prop:Main1} will follow directly from the estimates in Proposition \ref{prop:harmonic2}. In order to establish the more refined estimates in Proposition \ref{prop:Main3}, we need to exploit control on the oscillation of the harmonic functions $u^{\pm}$ in regions satisfying an exterior cone condition. 
\begin{proof}{Proposition \ref{prop:Main1}}
We fix the constant $c$ in the statement of the proposition to be as in Lemma \ref{lem:transverse}. Then, for $t\in S_c$, we can apply  Proposition \ref{prop:harmonic2} to the function $v_t^+ = t^{-1}v^+(tx)$. Since $v^+$ maximizes $u^+$ this implies that 
\begin{equation} \label{eqn:harmonic-3}
\begin{split}
\int_{B_t(\textbf{0})\cap\pa\Omega}(u^{+})^2\ud \sigma & \leq C \int_{\pa B_t(\textbf{0})\cap\Omega} (u^{+})^2 \ud \sigma, \\  (u^{+}({x}))^2 & \leq Ct^{-2}\int_{\pa B_t(\textbf{0})\cap\Omega} (u^{+})^2 \ud \sigma \quad\text{ for } x \in B_{t/2}(\textbf{0})\cap\Omega.
\end{split}
\end{equation}
Here and throughout $C$ is a constant depending only on the Lipschitz norm of $\pa\Omega$, which may change from line-to-line. By the definition of $M_{\textbf{0}}(t)$, we also have the estimates
\begin{equation} \label{eqn:harmonic-4}
\begin{split}
 \frac{\nu(x)\cdot x}{|x|^3}  & \leq 8t^{-3}{M}_{{\textbf{0}}}(t) \quad\text{ for } x\in \left(B_t(\textbf{0})\cap\pa\Omega \right)\backslash B_{t/2}(\textbf{0}) ,\\  
 \int_{B_{t/2}(\textbf{0})\cap\pa \Omega} \frac{\nu(x)\cdot x}{|x|^3} \ud \sigma  & \leq C\int_{0}^{t}s^{-2}{M}_{{\textbf{0}}}(s) \ud s.
\end{split}
\end{equation}
 Therefore, breaking the integral in the numerator of $A^{+}(t)$ into two pieces and using the estimates in \eqref{eqn:harmonic-3} and \eqref{eqn:harmonic-4}, we have
\begin{align} \label{eqn:harmonic-5}
A^{+}(t) &\leq C\left(t^{-1}M_{{\textbf{0}}}(t) +\int_{0}^{t}s^{-2}M_{{\textbf{0}}}(s) \ud s \right).
\end{align}
By definition $t^{-1}M_{{\textbf{0}}}(t)$ is certainly bounded by $1$, and by Assumption \ref{assum:Main}, the integral of $s^{-2}M_{\textbf{0}}(s)$ is bounded by a constant depending only on $C_*$. The estimate on $A^{-}(t)$ follows analogously, and so this proves Proposition \ref{prop:Main1}. 
\end{proof}
\begin{proof}{Proposition \ref{prop:Main3}}
To prove Proposition \ref{prop:Main3}, we need to control the integral of $A^{\pm}(t)/t$ for $t\in S_c$. Looking at the estimate on $A^{+}(t)$ in \eqref{eqn:harmonic-5} and using Assumption \ref{assum:Main}, we see that the reason this estimate is not sufficient is due to the term $\int_{0}^{t}s^{-2}M_{\textbf{0}}(s) \ud s$ appearing on the right hand side. Therefore, we need to sharpen the estimate on $u^{\pm}({x})$ for ${x}\in B_{t/2}(\textbf{0})$. We do this by obtaining a better estimate on $\left(u^{+}({x})\right)^2$ than in \eqref{eqn:harmonic-3}, and likewise for $\left(u^{-}(x)\right)^2$.
\begin{lem} \label{lem:harmonic-improved}
Let $s\in(0,\tfrac{1}{2}t)$ be given.  Recalling the decomposition of $S_c = \bigcup_{j\in J}I_j$, $I_j = [t_j,t_j')$, from \emph{Definition \ref{defn:dyadic}}, let $K(s)$ be the number of $j\in J_1 = J_1(c)$ for which $t_j\in(s,\tfrac{1}{2}t)$. Then, there exist constants $0<c_0<1$, $C>0$ $($depending only on the Lipschitz constant of $\pa\Omega$$)$ such that
\begin{align*}
\left(u^{\pm}({x})\right)^2 \leq Cc_0^{K(s)}t^{-2}\int_{\pa B_t(\etb)\cap\Omega} (u^{\pm})^2 \ud \sigma 
\end{align*}
for all ${x}\in B_{s}(\etb)\cap\Omega$. 
\end{lem}
Before proving Lemma \ref{lem:harmonic-improved}, let us first use it to complete the proof of Proposition \ref{prop:Main3}: Letting $j_1$, $j_2, \ldots$ be the increasing sequence of indices for which $j_k\in J_1$, we need to bound
\begin{align*}
\sum_{j\in J_1}\int_{I_j}\frac{A^{+}(t)}{t} \ud t =  \sum_{k}\int_{t_{j_k}}^{t_{j_k}'}\frac{A^{+}(t)}{t} \ud t .
\end{align*}
Using the estimate on $\left(u^{+}({x})\right)^2$ from Lemma \ref{lem:harmonic-improved} in place of the second estimate in \eqref{eqn:harmonic-3}, we obtain
\begin{align*}
A^{+}(t) \leq C\left(t^{-1}M_{\textbf{0}}(t) + \int_{0}^{t}s^{-2}c_0^{K(s)}M_{\textbf{0}}(s)\ud s \right). 
\end{align*}
We therefore have
\begin{align*}
  \sum_{j\in J_1}\int_{I_j}\frac{A^{+}(t)}{t} \ud t 
 \leq  C \sum_{j\in J_1}\int_{I_j}\left(t^{-2}M_{\textbf{0}}(t) + t^{-1}\int_{0}^{t}s^{-2}c_0^{K(s)}M_{\textbf{0}}(s) \ud s \right)\ud t.
\end{align*}
By Assumption \ref{assum:Main}, we are left to show that
\begin{align*}
 \sum_{j\in J_1}\int_{I_j}t^{-1}\left(\int_{0}^{t}s^{-2}c_0^{K(s)}M_{\textbf{0}}(s)\ud s \right)\ud t
\end{align*}
is bounded. Recalling that we have written the intervals $I_j$ with $j\in J_1$ by $[t_{j_k},t'_{j_k}]$, for $t\in[t_{j_k},t_{j_k}']$ and $s\in[t_{j_l},t_{j_{l-1}}]$ we have $K(s) = k-l$, and so we can write the above sum as
\begin{align*}
\sum_{k}\int_{t_{j_k}}^{t_{j_k}'}t^{-1}\left(\sum_{l:l>k}\int_{t_{j_l}}^{t_{j_{l-1}}}s^{-2}c_0^{k-l}M_{{\textbf{0}}}(s) \ud s \right)\ud t. 
\end{align*}
Swapping the order of integration, so that we compute the sum in $k$ first, since $c_0<1$, this is bounded by
\begin{align*}
C\sum_{l}\int_{t_{j_l}}^{t_{j_{l-1}}}s^{-2}M_{\textbf{0}}(s) \ud s,
\end{align*}
and so by Assumption \ref{assum:Main}, $A^{+}(t)$ satisfies the desired estimate, and likewise so does $A^{-}(t)$. We are left to prove Lemma \ref{lem:harmonic-improved}.
\\
\\
\textit{Proof of Lemma \ref{lem:harmonic-improved}:}
The functions $u^{\pm}({x})$ are harmonic inside the regions $\{u^{\pm}>0\}$, and satisfy Neumann boundary conditions on $B_1(\textbf{0})\cap\pa\Omega$. Therefore, reflecting across the convex boundary $\pa\Omega$, we obtain a function $\tilde{u}^{+}$ which is a solution of an elliptic equation in divergence form, with bounded measurable coefficients on a domain $\tilde{\Omega}^{+}$. Moreover, $\tilde{u}^{+}$ vanishes on the part of the boundary consisting of $\pa\{u^{+}>0\}$ and its reflection. For each $s\in(0,1)$ we also denote $\tilde{\Omega}_s^{+}$ to be the subset of $\tilde{\Omega}^{+}$ coming from the double of $B_s(\textbf{0})\cap\pa\Omega\cap\{u^+>0\}$. We define $\tilde{\Omega}^{-}$, $\tilde{\Omega}_s^{-}$, and $\tilde{u}^{-}$ analogously. 
\\
\\
Since $u^{\pm}(\textbf{0}) = 0$, with $\textbf{0}\in\pa\tilde{\Omega}_t^{\pm}$, if $\tilde{\Omega}_t^{\pm}$ were to both satisfy an exterior cone condition at $\textbf{0}$, then $u^{\pm}$ are H\"older continuous for some exponent $\gamma>0$, and the estimate in the statement of the lemma follows immediately. However, we have not established this exterior cone condition, and so we proceed as follows: Let $j\in J_1$, and consider the interval $I_j = [t_j,t_j']$. By the definition of $J_1$, we have $\Phi(t_j)\geq\Phi(t_j')$, and so in particular the integral of $\Phi'(t)/\Phi(t)$ over $I_j$ is non-positive. By Proposition \ref{prop:Main-half}, we therefore have
\begin{align} \label{eqn:harmonic-improved1}
\int_{I_j}t^{-1}(\beta^{+}(t) + \beta^{-}(t)-2)\ud t \leq0,
\end{align}
 with $\beta^{\pm}(t)$ the characteristic exponents from \eqref{eqn:Phi-deriv6}. Suppose that the measure of either $\tilde{\Omega}_{t_j'}^{+}$ or $\tilde{\Omega}_{t_j'}^{-}$ is less than $\eps_0 t_j^3$, for a small constant $\eps_0>0$. Then, by taking $\eps_0>0$ sufficiently small, the eigenvalues $\la^{\pm}(t)$ from Definition \ref{defn:DN} are sufficiently large for at least half of the values of $t\in I_j$ in order to contradict \eqref{eqn:harmonic-improved1}. Therefore, the sets $\tilde{\Omega}_{t_j'}^{+}$ and $\tilde{\Omega}_{t_j'}^{-}$ both have measure comparable to $t_j^3$. Applying Theorem 8.27 in \cite{GT} implies that there exists a constant $0<c_0<1$ such that $\left|\tilde{u}^{+}({x})\right| \leq  c_0 \max_{\tilde{\Omega}_{t_j'}^{+}}\left|\tilde{u}^{+}\right|$ for ${x}\in \tilde{\Omega}_{t_j}^{+}$, and likewise for $\tilde{u}^{-}(x)$. Repeating this estimate for each interval $I_j$ with $j\in J_1$ therefore proves the lemma. 
\end{proof}

\section{Lipschitz continuity of the minimizer $u$} \label{sec:Lipschitz}

In this section we will use the control on $\Phi(t)$ from Theorem \ref{thm:Phi} to prove Theorem \ref{thm:Lipschitz}, which we first state again:
\begin{thm} \label{thm:Lipschitz1}
 Suppose that $\Om$ is a convex domain, and let $u$ be a minimizer
of \eqref{eq:ACfunctional} among functions equal to $u_0$ on the compact set $K$.
After suitable dilation and translation, we may assume $0 \in \pa\Om$,
$\Om$ has diameter at least $2$, and $\bar B_2 \cap K = \emptyset$.  
 Under \emph{Assumption \ref{assum:Main}}, there exists a constant $C$, depending only on the $C^{\eps}$-norm of $u$ in $B_2 \cap \bar \Om$, 
 the Lipschitz constant of $\pa\Omega$, and the constant $C_*$ appearing in the  Dini condition, such that 
 \begin{align*}
 \sup_{\Omega\cap B_{1}}\left|\nabla u\right|\leq C.
 \end{align*}
\end{thm}
 Let $z \in \Omega\cap B_{1/2}$, and define $\rho({z})$ by  $\rho({z}) = \text{dist}({z},\{u=0\}).$  Since a harmonic function with normal derivative vanishing weakly on a convex domain is Lipschitz, we may assume that $\rho({z})$ is small enough that $B_{\rho({z})}({z}) \subset \Omega\cap B_1$. By otherwise rescaling around ${z}$ by
\begin{align*}
\tilde{u}({x}) = \rho({z})^{-1} u({z} + \rho({z})({x}-{z})),
\end{align*} 
we will assume that $\rho({z}) = 1$. Note that the Lipschitz norm is invariant under this rescaling. We want to use Theorem \ref{thm:Phi} to obtain a uniform upper bound on the size of $\left|u({z})\right|$. Let ${x}_0\in\Omega$ satisfy $u({x}_0)=0$ and $|{x}_0-{z}|=\rho({z}) = 1$. We will follow \lqu Another proof of Theorem 6.3' in \cite{ACF}, where the authors establish the Lipschitz continuity of the minimizer in the interior. To do so, they make use of the explicit harmonic measure on the ball, and so we require a domain containing ${x}_0$, which satisfies an equivalent harmonic measure estimate.
\begin{prop} \label{prop:L1-harmonic} 
There exist constants $c_1$, $C_1$, depending only on the Lipschitz norm of $\pa\Omega$,  such that for each ${x}_0\in \Omega\cap B_{1/2}$, we can form a  convex domain $\Omega_{{x}_{0}}\subset\Omega$, with the following properties:
\begin{enumerate}

\item[i)] The boundary of $\Omega_{{x}_{0}}$ consists of two parts $\pa\Omega_{{x}_{0},N}$ and $\pa\Omega_{{x}_{0},D}$. The first part is a (possibly empty) subset of $\pa\Omega$, and the second part ensures that $B_{c_1}({x}_{0}) \cap \Omega \subset \Omega_{{x}_{0}}$.

\item[ii)] Let $G$ be the Green's function for $\Omega_{{x}_{0}}$, with pole at ${x}_{0}$, with zero Dirichlet boundary conditions on $\pa\Omega_{{x}_{0},D}$, and zero Neumann boundary conditions (weakly) on $\pa\Omega_{{x}_{0},N}$. Then,
\begin{align*}
\norm{G}_{L^{\infty}\left(\Omega_{{x}_{0}}\backslash B_{c_1}({x}_0)\right)} \leq C_1, \qquad \norm{\nabla G}_{L^{\infty}\left(\pa\Omega_{{x}_{0},D}\right)} \leq C_1 .
\end{align*}

\end{enumerate}
\end{prop}
We will prove Proposition \ref{prop:L1-harmonic} in the Appendix, and first use it to complete the proof of Theorem \ref{thm:Lipschitz1}.
\begin{proof}{Theorem \ref{thm:Lipschitz1}}
We suppose that $u({z}) = M>0$, and obtain an upper bound on $M$ in terms of $C^{\eps}$-norm of $u$ in $B_2 \cap \bar \Om$,  the Lipschitz constant of $\pa\Omega$, and the constant $C_*$ appearing in the  Dini condition. In particular, all the constants in the proof below depend implicitly on the Lipschitz norm of $\pa\Omega$. Let ${x}_0\in \Omega$ satisfy $u({x}_0) = 0$,with $\left|{x}_0-{z}\right| = \rho({z}) = 1.$ We now form the domain $\Omega_{{x}_0}$ given in Proposition \ref{prop:L1-harmonic}.  Applying the Harnack inequality to the harmonic function $u^{+}({x})$ in $B_{1}({z})\cap\Omega$ (which applies to harmonic functions satisfying weak Neumann conditions on a convex boundary, \cite{Ra}), there exists a point ${y}$ in $\pa \Omega_{x_0}$, and  constants $\tilde{c}_1$, $\tilde{c}_2$, such that 
\begin{align*}
u({x}) > \tilde{c}_1M \text{ for } x \in B_{\tilde{c}_2}({y})\cap \Omega_{{x}_0}.
\end{align*}
Let $v$ be the harmonic function in the set $\Omega_{{x}_{0}}$, such that $v = u$ on $\pa\Omega_{{x}_{0},D}$, and $\pa_{\nu}v = 0$ weakly on $\pa\Omega_{{x}_{0},N}$. Then, following the proofs of Theorem 4.1 in \cite{ACF} (in the interior case in $n$ dimensions), and Lemmas 4 and 5 in \cite{GMR} (near a convex boundary in $2$ dimensions),  there exists a constant $C_2$ such that  
\begin{align} \label{eqn:upper-average}
0=u({x}_{0}) \leq v({x}_{0}) \leq C_2.
\end{align}
We also write $v = v^{+}-v^{-}$, with $v^{\pm}\geq0$ given to be harmonic in $\Omega_{{x}_{0}}$, such that $v^{\pm} = u^{\pm}$ on $\pa\Omega_{{x}_{0},D}$. In particular, we have $v^{+}({x}_0) \geq \tilde{c}_3M$, for a constant $\tilde{c}_3>0$. For $M$ sufficiently large, \eqref{eqn:upper-average} provides a lower bound on $v^{-}({x}_0)$ in terms of $v^{+}({x}_0)$ and implies that $v^{-}({x}_0) \geq \tfrac{1}{2}\tilde{c}_3M$. By the estimates on $\nabla G$ from Proposition \ref{prop:L1-harmonic}, we thus must have
\begin{align} \label{eqn:Lipthm1}
\int_{\pa \Omega_{{x}_0}}u^{-}(x) \ud \sigma({x}) \geq c^*_1M,
\end{align}
for a constant $c^*_1>0$. The Harnack inequality also ensures that there exists a point ${z}_0\in \Omega_{{x}_0}$, and constants $c_2^*$, $c_3^*>0$ such that 
\begin{align} \label{eqn:Lipthm2} 
u({x})>c_2^*M \text{ for } {x}\in B_{c_3^*}({z}_0) \subset  \Omega_{{x}_0}.
\end{align}
We will now combine \eqref{eqn:Lipthm1} and \eqref{eqn:Lipthm2} with Theorem \ref{thm:Phi} to obtain an upper bound on $M$.  Using polar coordinates $(\rho,\omega)$ in $\R^n$, with origin at ${z}_0$, let $\Gamma$ be the set of directions for which $u(x)<0$ at some point $x\in\Omega_{x_0}$ on the ray emanating from $z_0$.  Then, using \eqref{eqn:Lipthm1} we have
\begin{align} \label{eqn:Lipthm3}
c^*_1M \leq \int_{\pa \Omega_{{x}_0}} u^{-}({x}) \ud\sigma({x})=  \int_{\Gamma} \left(\int \pa_{\rho}u^{-}({x}) \ud \rho\right) \ud \sigma(\omega) \leq C_1\left|\Gamma\right|^{1/2}\left(\int_{B_{1}({x}_0)\cap\Omega}\left|\nabla u^{-}\right|^2  \ud x\right)^{1/2},
\end{align}
where we have used that $u^{-}(x) = 0$ in $B_{c_3^*}({z}_0)$. Using \eqref{eqn:Lipthm2} and the definition of $\Gamma$ we also have
\begin{align} \label{eqn:Lipthm4}
c_2^*\left|\Gamma\right|M \leq -\int_{\Gamma} \left(\int \pa_{\rho}u^{+}({x}) \ud \rho\right) \ud \sigma(\omega) \leq C_1\left|\Gamma\right|^{1/2}\left(\int_{B_{1}({x}_0)\cap\Omega}\left|\nabla u^{+}\right|^2  \ud x\right)^{1/2}.
\end{align}
 Multiplying the inequalities in \eqref{eqn:Lipthm3} and \eqref{eqn:Lipthm4}, and using Theorem \ref{thm:Phi} then gives an upper bound on $M$.
\\
\\
In the case, where $u({z})<0$, we can proceed in precisely the same way with $v^{+}$ and $v^{-}$ interchanged, and again use \eqref{eqn:upper-average} to obtain a lower bound on $v^{+}({x}_0)$ in terms of $v^{-}({x}_0)$. 
\\
\\
Since $u$ is harmonic in $\{u>0\}$ and $\{u<0\}$, the upper bound on $|u({z})|$ of $\left|u({z})\right| \leq C\rho({z})$, with $\rho({z}) = \text{dist}\left({z},\{u=0\}\right)$ implies the Lipschitz continuity statement in Theorem \ref{thm:Lipschitz}.
\end{proof}

\section{An example of failure of the Dini condition on the convex body in dimensions $3$ and higher} \label{sec:assum}

In this section we will discuss the Dini condition, Assumption \ref{assum:Main}, placed on $\Omega$. Recalling the definition  
 $M_{{x}}(t)$ from Definition \ref{defn:N}, Assumption \ref{assum:Main} states that
\begin{align*}
\int_{0^{+}}^{1}\frac{M_{{x}}(t)}{t^2}\ud t < C^*,
\end{align*}
uniformly over points ${x}\in\pa\Omega\cap B_1$. In this section, we will establish the following three statements:
\begin{prop} \label{prop:Dini1}
Suppose that $\Omega$ is a convex $n$-dimensional domain, such that $\pa\Omega$ has $C^{1,\beta}$-regularity. Then, the Dini condition on $t^{-1}M_{x}(t)$, \emph{Assumption \ref{assum:Main}}, holds. 
\end{prop}
\begin{prop} \label{prop:Dini2}
Suppose that $\Omega$ is a convex $2$-dimensional domain. Then, the Dini condition on $t^{-1}M_{x}(t)$, \emph{Assumption \ref{assum:Main}}, holds. 
\end{prop}
\begin{prop} \label{prop:Dini3}
There exists convex domains $\Omega$ in $3$ and higher dimensions for which the Dini condition on $t^{-1}M_{x}(t)$, \emph{Assumption \ref{assum:Main}}, fails. Moreover, the condition fails in the sense that $t^{-2}M_{x}(t)$ fails to be integrable at a point $x\in\pa\Omega$.
\end{prop}
Before proving the propositions, we first note another expression that controls $M_x(t)$ for some points $x$ on the boundary of $\Omega$. Setting the base point $x$ to be the origin  $\textbf{0}$, suppose that $\textbf{0}$ is at the centre of a Lipschitz parameterization of $\pa\Omega$, in the sense that we can write $\pa \Omega$ as the graph of a convex function $x_n = f(x')\geq0$ for all $|x'|\leq 1$ in $\R^{n-1}$. 
\begin{defn} \label{defn:N-defn}
Let $f$ be the convex function as above. At points where $f$ is differentiable define
\begin{align} \label{eqn:N-defn}
N_{\etb}(t,\theta) = \left(1+|\nabla f(t,\theta)|^2\right)^{-1/2}\left\{t\pa_tf(t,\theta) - f(t,\theta)\right\},
\end{align}
where $(t,\theta)$ are polar coordinates in $\R^{n-1}$.
\end{defn}
 The function $N_{\textbf{0}}$ equals $\nu(x(t,\theta))\cdot x(t,\theta)$ at the point $x(t,\theta) = (x'(t,\theta),f(x'(t,\theta)))$. Since $t\pa_tf(t,\theta)-f(t,\theta)$ is increasing in $t$ (using the convexity of $f$), we have the relations
\begin{align} \label{eqn:N-ineq}
M_{\textbf{0}}(t) \leq \sup_{\theta\in \mathbb{S}^{n-2}} N_{\textbf{0}}(t,\theta) \leq M_{\textbf{0}}(Ct),
\end{align}
for a constant $C$ depending only on the Lipschitz norm of $f$. In particular, for $f\in C^{1,\beta}$, then $|M_{\textbf{0}}(t)|\leq Ct^{1+\beta}$, and the Dini assumption holds at this point. When $\Omega$ is $C^{1,\beta}$ any point on $\pa\Omega$ can be placed at the centre of such a Lipschitz parameterization (after a rescaling), this establishes Proposition \ref{prop:Dini1}.
\\
\\
We now turn to studying the Dini condition for a general convex domain. For any convex domain, we immediately see that $|M_{\textbf{0}}(t)|\leq Ct$, but this (just) misses ensuring that $t^{-2}M_{\textbf{0}}(t)$ is integrable. Our example of a convex body $\Omega$ in three and higher dimensions for which the Dini condition does not hold  is intricate, and exploits the fact that $M_{\textbf{0}}(t)$ measures the supremum of $N_{\textbf{0}}(t,\theta)$ over $\theta$, and as $t$ decreases, different angles $\theta$ maximize $N_{\textbf{0}}(t,\theta)$. To see why this must be the case for any such example where Assumption \ref{assum:Main} fails, we first consider the average of $N_{\textbf{0}}(t,\theta)$ in $\theta$.
\begin{prop} \label{prop:Dini-average}
Letting $N_{\etb}(t,\theta)$ be as in \emph{Definition \ref{defn:N-defn}}, define $\bar{M}_{\etb}(t)$ by $\displaystyle{\bar{M}_{\etb}(t) = \int_{\mathbb{S}^{n-2}}N_{\etb}(t,\theta) \ud \sigma_{n-2}}$. Then, there exists a constant $\bar{C}_*$, depending only on the Lipschitz constant of $f$ such that
\begin{align*}
\int_{0^+}^{1}\frac{\bar{M}_{\etb}(t)}{t^2} \ud t < \bar{C}_*.
\end{align*}
\end{prop}
\textit{Proof of Proposition \ref{prop:Dini-average}:}
We first note that since $t\pa_tf(t,\theta)-f(t,\theta)$ is increasing in $t$ (for any fixed $\theta$), we have the inequalities
\begin{align*}
0\leq \frac{t}{2} \int_{t/2}^{t}s^{-2}(s\pa_sf(s,\theta) - f(s,\theta)) \ud s \leq t\pa_tf(t,\theta)-f(t,\theta) \leq 4t \int_{t}^{2t}s^{-2}(s\pa_sf(s,\theta) - f(s,\theta)) \ud s.
\end{align*}
Computing the integrals on either side of these inequalities gives
\begin{align*}
\frac{1}{2}\left(f(t,\theta)-f(t/2,\theta)\right) \leq t\pa_tf(t,\theta)-f(t,\theta) \leq 2\left(f(2t,\theta)-f(t,\theta)\right).
\end{align*}
The proof of the proposition therefore follows from the lemma below by writing 
\begin{align*}
\int_{0^+}^{c}\frac{\bar{M}_{\textbf{0}}(t)}{t^2} \ud t  = \int_{\mathbb{S}^{n-2}} \left\{\int_{0^+}^{c} t^{-2}N_{\textbf{0}}(t,\theta) \ud t \right\} \ud \sigma_{n-2}.
\end{align*} 
\begin{lem} \label{lem:1d}
Let $g(t)\geq0$ be a convex function for $0 \leq t \leq 1$, with $g(0)=0$, Lipschitz constant $C_{g}$. Then, setting $h(t) = g(t) - 2g(t/2)$, we have
\begin{align*}
\int_{0^+}^{1} t^{-2}h(t) \ud t \leq 4C_{g}.
\end{align*}
\end{lem}
\begin{proof}{Lemma \ref{lem:1d}}
For each $N\geq0$, we have
\begin{align*}
g(1) = h(1) + 2g(1/2) = h(1) + 2h(1/2) + 4g(1/4) = \sum_{j=0}^{N} 2^{j}h(2^{-j}) + 2^{N+1}g(2^{-N-1}).
\end{align*}
Moreover, by convexity $h(t)$ is nonnegative and an increasing function of $t$. Thus,
\begin{align*}
\int^{1}_{2^{-N-1}} t^{-2}h(t) \ud t = \sum_{j=0}^{N}\int_{2^{-j-1}}^{2^{-j}}t^{-2}h(t) \ud t \leq \sum_{j=0}^{N} 2^{-j}2^{2j+2}h(2^{-j}) \leq 4\left(g(1) -2^{N+1}g(2^{-N-1})\right) \leq 4C_{g}.
\end{align*}
Letting $N$ tend to infinity finishes the proof of the lemma. 
\end{proof}
We can also use Proposition \ref{prop:Dini-average} to show that Assumption \ref{assum:Main} is always satisfied for two dimensional convex domains, and hence prove Proposition \ref{prop:Dini2}: Since $\mathbb{S}^{n-2}$ only consists of two points when $n=2$, Proposition \ref{prop:Dini-average} show that the Dini condition in Assumption \ref{assum:Main} is satisfied at any point $\textbf{0}\in \pa\Omega$, with $\Omega \subset \R^2$ convex, and $\textbf{0}$ at the centre of the Lipschitz parameterization. To show that Assumption \ref{assum:Main} holds for a general point $x$ on the boundary, we proceed as follows. For $t\leq \tfrac{1}{2}|x|$, $\pa\Omega\cap B_t(x)$ can be written as the graph of a convex function, with the Lipschitz parameterization centred at $x$, and so by the discussion above
\begin{align*}
\int_{0^{+}}^{\tfrac{1}{2}|x|} \frac{M_{x}(t)}{t^2}\ud t \leq C.
\end{align*}
Since $\nu(y)\cdot(y-x) -\nu(y)\cdot y \leq |x|$, we have $M_{x}(t) \leq M_{\textbf{0}}(\tfrac{3}{2}t) + |x|$  for $t\geq\tfrac{1}{2}|x|$. Moreover, $t^{-2}|x|$ is integrable for $t\geq\tfrac{1}{2}|x|$, and so by our estimates on $M_{\textbf{0}}(t)$, the function $t^{-2}M_{x}(t)$ is integrable. Thus, Assumption \ref{assum:Main} always holds in two dimensions, and this proves Proposition \ref{prop:Dini2}.
\begin{rem} \label{rem:Dini-average}
Recall  the  spherical slices $V_t$ of $\Omega$ from \emph{Definition} \ref{defn:V}. The estimate on $\bar{M}_{\etb}(t)$ in \emph{Proposition \ref{prop:Dini-average}} ensures that
\begin{align*}
\int_{0^+}^{1} \frac{\emph{Area}_{\mathbb{S}^{n-1}}(V_{t/2}\backslash V_t)}{t^2} \ud t <\infty.
\end{align*}
However, when the limiting slice $V_0$ contains antipodal points, the convex hull of $V_{t}$ in $\mathbb{S}^{n-1}$ is not necessarily contained in $V_{\alpha t}$ for any fixed $0<\alpha<1$. Therefore, the estimate in \emph{Proposition} \ref{prop:Dini-average} is not sufficient to prove \emph{Proposition} \ref{prop:convex-approx} concerning the existence of a geodesically convex set $W_{t}\subset\mS^{n-1}$ which is sufficiently close in area to the spherical slice $V_{t}$.
\end{rem}
We end this section by proving Proposition \ref{prop:Dini3} by exhibiting the example of a convex domain in dimensions three and higher for which the Dini condition, Assumption \ref{assum:Main}, does not hold. We just consider focus on the three dimensional case, since for higher dimensions we can take the cartesian product of the three dimensional domain with copies of $\mathbb{R}$. We will prove the following.
\begin{prop} \label{prop:Dini4}
Let $C_1^*\geq0$ be given, and let $(t,\theta)$ be polar coordinates centred at $(0,0)\in\R^2$. Then,  there exists a bounded, convex function $f(t,\theta)$ on $B_{1/2}(\etb)\subset\R^2$, with $f(0,0) =0$, $f(t,\theta) \geq 0$, and with the following properties:
\begin{enumerate}
\item[i)] The limit $\lim_{t\to0} \frac{f(t,\theta)}{t} = C_1^*$ holds for all $\theta$.
\item[ii)] There exists an absolute constant $c^*>0$ such that for all $t\in(0,\tfrac{1}{2})$, we have
\begin{align*}
c^* \leq \frac{\sup_{\theta}N(t,\theta)}{\frac{t}{\log(t^{-1})}},
\end{align*}
where $N(t,\theta) =t\pa_tf(t,\theta) - f(t,\theta)$. In particular, 
\begin{align*}
\int_{\delta}^{1/2} \frac{\sup_{\theta}N(t,\theta)}{t^2} \ud t
\end{align*}
tends to $\infty$ as $\delta$ tends to $0$.
\end{enumerate}
\end{prop}
Letting $\Omega$ be the three dimensional convex domain consisting of the graph of $f$, and using the definition of $N_{\textbf{0}}(t,\theta)$ from Definition \ref{defn:N-defn} and its relation to $M_{\textbf{0}}(t,\theta)$ in \eqref{eqn:N-ineq}, we see that Proposition \ref{prop:Dini4} implies the failure of the Dini condition in Proposition \ref{prop:Dini3}.
\begin{proof}{Proposition \ref{prop:Dini4}}
To construct the function $f$, we will split the range of $t$ into a union of intervals $[t_{k},t_{k-1}]$ for an appropriately chosen sequence $t_{k}$ tending to $0$. For $t$ in each of these intervals, we also break up $\theta\in[0,2\pi]$ into intervals $I_{j,k}$, and associate values $t_{j,k}\in[t_{k},t_{k-1}]$, $\theta_{j,k}\in I_{j,k}$ to each of these intervals. We then define  $f$ so that for all $j$ and $k$,
\begin{align*}
N(t_{j,k},\theta_{j,k}) \geq c^* \frac{t_{j,k}}{\log(t_{j,k}^{-1})},
\end{align*}
while ensuring that $f$ is convex.  Let us first set some notation. 
\begin{defn} \label{defn:Main}
Fix $\alpha_1$ such that $\frac{1}{7}< \alpha_1 - 1 < \frac{1}{6}$, and a large integer $M_1$ (to be specified below). For each $k\geq1$, define $t_k$, $\delta_k$ and $\theta_k$ via
\begin{align*} 
t_k= e^{-\lf k^{\alpha_1}\rf}, \qquad \delta_k = \left(\log(t_k^{-1})\right)^{-1} = \left(\lf k^{\alpha_1}\rf\right)^{-1}, \qquad e^{1/\theta_k} = \frac{t_{k-1}}{t_k}e^{-M_1}. 
\end{align*}
\end{defn}
This definition ensures that $\theta_{k}^{-1}$ is an integer and that  $k^{\alpha_1-1}\theta_{k}$ is uniformly bounded above and away from $0$ (by constants independent of $M_1$ for $k$ sufficiently large). 
\begin{claim} \label{claim:Main}
To prove \emph{Proposition \ref{prop:Dini4}} it is sufficient to construct a function $N(t,\theta)$ with the following properties for all $k\geq K$, for some constant $K$:
\begin{pro} \label{prop:1}
$N(t_k,\theta) = t_k \delta_k, \qquad  N(t_{k-1},\theta) = t_{k-1}\delta_{k-1}$.
\end{pro}
\begin{pro} \label{prop:2}
$\sup_{\theta}N(t,\theta) \geq \tilde{c}t \delta_{k}$ for all $t_k \leq t \leq t_{k-1}$, for a constant $\tilde{c}>0$.
\end{pro}
\begin{pro} \label{prop:3}
$\int_{t_{k}}^{t_{k-1}} \frac{N(t,\theta)}{t^2} \ud t = \tilde{C}\delta_k$, for a constant $\tilde{C}$ bounded above and away from $0$.
\end{pro}
\begin{pro} \label{prop:4}
The bounds
\begin{align*}
\pa_{t}N(t,\theta) \geq c_1\frac{N(t,\theta)^2}{\theta_k^6t^2}, \qquad \left|\pa_{\theta}N(t,\theta)\right| \leq C_1\frac{N(t,\theta)}{\theta_k^2}, \qquad \left|\pa_{\theta}^2N(t,\theta)\right| \leq C_1\frac{N(t,\theta)}{\theta_k^5}
\end{align*} 
hold, for constants $c_1$, $C_1$, for all $t_{k} \leq t \leq t_{k-1}$, $0 \leq \theta \leq 2\pi$.
\end{pro}
Here $\tilde{c}$, $\tilde{C}$, $c_1$, $C_1$ etc., denote constants that are independent of $k$. 
\end{claim}
\begin{proof}{Claim \ref{claim:Main}}
By Property \ref{prop:2}), we have the desired lower bound
\begin{align*}
\sup_{\theta}N(t,\theta) \geq \tilde{c} t \delta_{k} \geq c^* t\left(\log(t^{-1})\right)^{-1},
\end{align*} 
for $0<t\leq t_{K-1}$. Since we can rescale any convex function on $t \leq t_{K-1}$ to the ball $B_{1/2}$, it is therefore sufficient to show that Properties \ref{prop:1}) to \ref{prop:4}) allow us to define a convex function $f(t,\theta)$ for $0 < t \leq t_{K-1}$. For a given $C_1^*\geq0$, we define $f(t,\theta)$ by
\begin{align} \label{eqn:fdef}
\frac{f(t,\theta)}{t} = C_1^* + \int_{0}^{t} \frac{N(\tau,\theta)}{\tau^2} \ud \tau,
\end{align}
This integral converges for all $\theta$ by Property \ref{prop:3}) above (since $\delta_k= \left(\lf k^{\alpha_1}\rf\right)^{-1}$, with $\alpha_1>1$), and so in particular $f$ is bounded, and satisfies
\begin{align*}
\lim_{t\to0}\frac{f(t,\theta)}{t} = C_1^*.
\end{align*} 
Let us check that this definition gives a convex function:  By Property \ref{prop:3}), for $t_k \leq t \leq t_{k-1}$, we have
\begin{align} \label{eqn:fdef1}
\pa_{t}f(t,\theta) \geq C_1^* + \int_{0}^{t_k} \frac{N(\tau,\theta)}{\tau^2} \ud \tau  \geq C_1^* +\tilde{C} \sum_{j\geq k}\delta_k \geq c\theta_k,
\end{align}
for any $C_1^*\geq0$. In polar coordinates the Hessian of $f$ is written as
$$
\begin{pmatrix}
\pa_{t}^2f & \pa_{\theta}\pa_tf - t^{-1}\pa_{\theta} f \\
 \pa_{\theta}\pa_tf - t^{-1}\pa_{\theta} f & t\pa_t f + \pa_{\theta}^2 f
\end{pmatrix} = \begin{pmatrix}
t^{-1}\pa_{t}N & t^{-1}\pa_{\theta}N \\
t^{-1}\pa_{\theta}N & t\pa_t f + \pa_{\theta}^2 f
\end{pmatrix} .
$$ 
By Properties \ref{prop:3}) and \ref{prop:4}), we have
\begin{align*}
\left| \pa_{\theta}^2f(t,\theta)\right| \leq  t\int_{0}^{t} \frac{\left|\pa^2_{\theta}N(\tau,\theta)\right|}{\tau^2} \ud \tau \leq C_1\theta_{k}^{-5}t\int_{0}^{t} \frac{N(\tau,\theta)}{\tau^2} \ud \tau \leq C_1\theta_{k}^{-5}t\cdot\tilde{C}\delta_k = C_1\tilde{C}t\theta_{k}^{-5}\delta_{k}.
\end{align*}
In particular, combining this with \eqref{eqn:fdef1} and using $\alpha_1-1< \tfrac{1}{6}$, we have $t\pa_t f + \pa_{\theta}^2 f \geq c^*\theta_{k}t>0$ for $t_k \leq t \leq t_{k-1}$, for some absolute constant $c^*$, and $k\geq K$, with $K$ sufficiently large. Thus,   the determinant of the Hessian is bounded below by
\begin{align*}
c^*\theta_k\pa_{t}N(t,\theta) - \frac{\left(\pa_{\theta}N(t,\theta)\right)^2}{t^2}.
\end{align*}
For $k\geq K$ sufficiently large, and using Property \ref{prop:4}), this determinant can be made positive, ensuring that $f(t,\theta)$ is a convex function. This therefore concludes the proof of the claim. 
\end{proof}
Let us briefly summarise the construction of $N(t,\theta)$:
\begin{defn} \label{defn:interval}
Break up $[0,\pi]$ into $\theta_k^{-1}$ equal intervals $I_{j,k} = [j\pi \theta_k,(j+1)\pi\theta_k]$, $0 \leq j \leq \theta_k^{-1}-1$, and define $\theta_{j,k}^{\pm} =  j\pi \theta_k \pm \theta_k^{3}$. 
\end{defn}
We construct $N(t,\theta)$  by first defining $N(t,\theta)$ for $\theta =\theta_{j,k}^{\pm}$, and $t_k \leq t \leq t_{k-1}e^{-M_1}$. We define $N(t,\theta)$ for these angles and this range of $t$ so that $N(t,\theta_{j,k}^{+}) = N(t,\theta_{j,k}^{-}),$ and such that Properties \ref{prop:1}) and  \ref{prop:2}), together with the first inequality in Property \ref{prop:4}) hold. We then interpolate between these values of $\theta$ so that the final two inequalities in Property \ref{prop:4}) hold, together with $N(t,\theta)$ being able to be extended evenly across the $x$-axis. Moreover, the definition will ensure that
\begin{align*}
\int_{t_k}^{t_{k-1}e^{-M_1}} \frac{N(t,\theta)}{t^2} \ud t = F_{k}(\theta)\delta_k,
\end{align*}
for a function $F_{k}(\theta)$, satisfying
\begin{align*}
C_2^{-1} \leq F_{k}(\theta) \leq C_2, \qquad \left|\pa_{\theta}F_{k}(\theta)\right| \leq C_2\theta_{k}^{-2}, \qquad \left|\pa_{\theta}^2F_{k}(\theta)\right| \leq C_2\theta_{k}^{-5}. 
\end{align*}
The constants above will be independent of $M_1$, and so this will then allow us to choose $M_1$ and  define $N(t,\theta)$ for $0 \leq \theta \leq 2\pi$, $t_{k-1}e^{-M_1} \leq t \leq t_{k-1}$ in such a way that Property \ref{prop:3}) is now satisfied, together with Properties \ref{prop:1}), \ref{prop:2}), and \ref{prop:4}) continuing  to hold. In the rest of this section we describe this construction of $N(t,\theta)$ in detail. 
\\
\\
\textbf{Definition of $N(t,\theta)$ for $\theta = \theta_{j,k}^{\pm}$}: 
\\
\\
We first define $N(t,\theta)$ for $t_k \leq t \leq t_{k-1}e^{-M_1}$ and $\theta = \theta_{j,k}^{\pm}$. Roughly speaking, this definition ensures that $N(t,\theta)$ stays comparable to $t_k\delta_{k}$ until $t$ reaches a certain scale $t_{j,k}$ (to be defined below), where it increases to $t_{j,k}\delta_{k}$. It then stays comparable to $t_{j,k}$, until $t$ reaches $t_{k-1}e^{-M_1}$ where it increases to $t_{k-1}e^{-M_1}\delta_{k-1}$.
\begin{defn} \label{defn:tjk}
Set 
\begin{align*}
t_{j,k} = t_{k}e^{j}, \qquad t_{j,k}^{(1)} = \max\left\{t_{k},t_{k}^{1/2}t_{k,j}^{1/2}\delta_{k}^{1/2}\theta_{k}^{-3}\right\}, \qquad t_{j,k}^{(2)} = \min\left\{t_{j,k}\theta_k^{-1/2},t_{k-1}e^{-M_1-1}\right\}
\end{align*}
\end{defn}
\begin{defn} \label{defn:N1}
We define $N(t,\theta)$ for $\theta = \theta^{\pm}_{j,k}$, and $t_k \leq t \leq t_{k-1}e^{-M_1}$ by
$$
N(t,\theta^{\pm}_{j,k}) = 
\begin{cases}
 t_k\delta_{k}\left(1+\frac{\delta_{k}}{\theta_k^{6}t}(t-t_k)\right), & \quad  t_k \leq  t \leq  t_{j,k}^{(1)} \\
 N(t_{j,k}^{(1)},\theta^{\pm}_{j,k}) + \frac{t_k\delta_k}{t_{j,k}}(t-t_{j,k}^{(1)}), & \quad  t_{j,k}^{(1)} \leq t \leq t_{j,k} \\
 N(t_{j,k},\theta^{\pm}_{j,k})e^{(t-t_{j,k})/(\alpha_{j,k}t_{j,k})}, & \quad  t_{j,k} \leq t \leq et_{j,k} \\
 \delta_k\theta_k^{1/2}t -\beta_{j,k}^{(1)}\left(\frac{t^2}{2} - et_{j,k}t\right) + \beta_{j,k}^{(2)}, & \quad  et_{j,k} \leq t \leq t_{j,k}^{(2)}\\
  N(t_{j,k}^{(2)},\theta^{\pm}_{j,k})\left(1+\frac{\delta_k}{\theta_k^6t}(t-t_{j,k}^{(2)})\right), & \quad  t_{j,k}^{(2)} \leq t \leq t_{k-1}e^{-M_1-1}\\
 N(t_{k-1}e^{-M_1-1},\theta^{\pm}_{j,k}) + \gamma_{j,k}(t-t_{k-1}e^{-M_1-1}), & \quad  t_{k-1}e^{-M_1-1} \leq t \leq t_{k-1}e^{-M_1}.
\end{cases}
$$
\end{defn}
In the above we have set
\begin{align*} 
\alpha_{j,k} & = (e-1)\left(\log\left(\frac{t_{j,k}\delta_ke}{N(t_{j,k},\theta_{j,k}^{\pm})}\right)\right)^{-1} \\ 
\beta_{j,k}^{(1)} & = (t_{j,k}^{(2)} - t_{j,k}e)^{-1}\left(\delta_{k}\theta_{k}^{1/2} - \frac{\delta^2_{k}t_{j,k}}{\theta_{k}^6t_{j,k}^{(2)}}\right) \\ 
\beta_{j,k}^{(2)} & = t_{j,k}\delta_ke -  \delta_k\theta_k^{1/2}t_{j,k}e + \beta_{j,k}^{(1)}e^2t_{j,k}^2  \\ \gamma_{j,k} & = \delta_{k-1}\frac{e}{e-1} - \frac{e^{M_1+1}}{t_{k-1}(e-1)}N(t_{k-1}e^{-M_1-1},\theta_{j,k}^{\pm}). 
\end{align*}
Note that the choice of $\alpha_{j,k}$, $\beta_{j,k}^{(i)}$ and $\gamma_{j,k}$ ensures that $N(t_{j,k}e,\theta^{\pm}_{j,k}) = t_{j,k}\delta_ke$, as well as $\pa_{t}N(t_{j,k}^{(2)}, \theta_{j,k}^{\pm})$ being comparable from both sides to $\frac{\delta^2_{k}t_{j,k}}{\theta_{k}^6t_{j,k}^{(2)}}$, the value of $N(t_{j,k}^{(2)},\theta_{j,k}^{\pm})$ is comparable to $\delta_kt_{j,k}$, and $N(t_{k-1}e^{-M_1},\theta^{\pm}_{j,k}) = e^{-M_1}t_{k-1}\delta_{k-1}$. Moreover, it ensures that $N(t,\theta)$ is continuous in $t$.
\\
\\
By definition, we have
\begin{align*}
N(t_{k},\theta^{\pm}_{j,k}) = t_{k}\delta_{k}, \qquad N(et_{j,k},\theta^{\pm}_{j,k}) = et_{j,k}\delta_{k}, \qquad
C^{-1}\delta_k \leq \int_{t_k}^{t_{k-1}e^{-M_1}} \frac{N(t,\theta^{\pm}_{j,k})}{t^2} \ud t  \leq C\delta_{k},
\end{align*}
for a constant $C$  independent of $j$, $k$ and $M_1$. By computing $\pa_tN(t,\theta_{j,k}^{\pm})$ and using $\alpha_1-1<\tfrac{1}{6}$ (ensuring that $\delta_k\theta_{k}^{-7}$ is bounded by a constant independent of $k$), the first inequality in Property \ref{prop:4}) also holds for these values of $\theta$ and range of $t$. The $t$-values $t_{j,k}^{(1)}$ and $t_{j,k}^{(2)}$ are introduced in Definition \ref{defn:N1} in order to decrease the difference in the magnitude of $\pa_tN(t,\theta^{\pm}_{j,k})$ between $t\in [t_{j,k},et_{j,k}]$ and nearby values of $t$, while still ensuring that Property \ref{prop:4}) holds. This is important to ensure that Property \ref{prop:4}) continues to hold when we interpolate in $\theta$ in Definition \ref{defn:N3} below. 
\\
\\
\textbf{Definition of $N(t,\theta)$ for $\theta_{j,k}^{+} \leq \theta \leq \theta_{j+1,k}^{-}$}:
\\
\\
To define $N(t,\theta)$ for $t_{k} \leq t \leq t_{k-1}e^{-M_1}$, $\theta_{j,k}^{+} \leq \theta \leq \theta_{j+1,k}^{-}$ we interpolate  between the values of $N(t,\theta_{j,k}^{+})$ and $N(t,\theta_{j+1,k}^{-})$:  For each $\theta$ in this range, define $a_{j,k} = a_{j,k}(\theta)$ so that
\begin{align} \label{eqn:ajk}
\theta = (1-a_{j,k})\theta_{j,k}^{+} + a_{j,k}\theta_{j+1,k}^{-1}, \qquad a_{j,k} = \frac{\theta - \theta_{j,k}^{+}}{\theta_{j+1,k}^{-} - \theta_{j,k}^{+}}. 
\end{align}
This ensures that 
\begin{align} \label{eqn:ajkderiv}
C^{-1}\theta_k^{-1} \leq\pa_{\theta} a_{j,k}(\theta) \leq C\theta_k^{-1}.
\end{align}
uniformly in $j$ and $k$. 
\begin{defn} \label{defn:N2}
For  $t_{k} \leq t \leq t_{k-1}e^{-M_1}$, $\theta_{j,k}^{+} \leq \theta \leq \theta_{j+1,k}^{-}$, we define $N(t,\theta)$ by
\begin{align*}
N(t,\theta) = N(t,\theta_{j,k}^{+})^{1-a_{j,k}}N(t,\theta_{j+1,k}^{-})^{a_{j,k}} = N(t,\theta_{j,k}^{+} )\exp\left(a_{j,k} \log\left(\frac{N(t,\theta_{j+1,k}^{-})}{N(t,\theta_{j,k}^{+})}\right)\right). 
\end{align*}
\end{defn}
Since $N(t,\theta)$ interpolates between $N(t,\theta_{j,k}^{+})$ and $N(t,\theta_{j+1,k}^{-})$, by Definition \ref{defn:N1}, we immediately see that for this range of $t$ and $\theta$ we have
\begin{align*}
N(t_k,\theta) = t_k\delta_k, \qquad  N(t_{k-1}e^{-M_1},\theta) = e^{-M_1}t_{k-1}\delta_{k-1}, \qquad 
C^{-1} \delta_k \leq \int_{t_k}^{t_{k-1}e^{-M_1}} \frac{N(t,\theta)}{t^2} \ud t  \leq C\delta_k.
\end{align*}
It also ensures the bounds
\begin{align*}
\left|\pa_{\theta}N(t,\theta)\right| & =  \left|\log\left(\frac{N(t,\theta_{j+1,k}^{-})}{N(t,\theta_{j,k}^{+})}\right)\pa_{\theta}a_{j,k}N(t,\theta)\right| \leq C\theta_k^{-1}\left|\log\left(t_{k-1}/t_k\right)\right|N(t,\theta) \leq C\theta_k^{-2}N(t,\theta), \\
\left|\pa_\theta^2N(t,\theta)\right|  & \leq C\theta_k^{-4}N(t,\theta).
\end{align*} 
Combining this with computing a lower bound for $\pa_tN(t,\theta)$ in terms of  $\pa_{t}N(t,\theta_{j,k}^{+})$ and $\pa_{t}N(t,\theta_{j+1,k}^{-})$, ensures that Property \ref{prop:4}) holds for $t_{k} \leq t \leq t_{k-1}e^{-M_1}$, $\theta_{j,k}^{+} \leq \theta \leq \theta_{j+1,k}^{-}$.
\\
\\
\textbf{Definition of $N(t,\theta)$ for $\theta_{j,k}^{-} \leq \theta \leq \theta_{j,k}^{+}$}:
\\
\\
We now define $N(t,\theta)$ for $t_{k} \leq t \leq t_{k-1}e^{-M_1}$ and $\theta$ in the interval $\theta_{j,k}^{-} \leq \theta \leq \theta_{j,k}^{+}$ of length $2\theta_k^3$. 
\begin{defn} \label{defn:N3}
For each $j$, $1 \leq j \leq \theta_k^{-1}-1$, and $\theta$  in the interval $\theta_{j,k}^{-} \leq \theta \leq \theta_{j,k}^{+}$, we make the definition 
\begin{align*}
N(t,\theta) = A_{j,k}(t) (\theta-\theta_{j,k}^{-})^3 + B_{j,k}(t) (\theta-\theta_{j,k}^{-})^2 + C_{j,k}(t)(\theta-\theta_{j,k}^{-}) + N(t,\theta_{j,k}^{-}),
\end{align*}
with 
\begin{align*}
(2\theta_k^3)^2 A_{j,k}(t) & =  \pa_{\theta}N(t,\theta_{j,k}^{-}) + \pa_{\theta}N(t,\theta_{j,k}^{+}),\\ 2\theta_k^3B_{j,k}(t)&   = -2 \pa_{\theta}N(t,\theta_{j,k}^{-})- \pa_{\theta}N(t,\theta_{j,k}^{+}), \\ C_{j,k}(t)  & = \pa_{\theta}N(t,\theta_{j,k}^{-}).
\end{align*}
For the same range of $t$ and $0\leq \theta \leq \theta_k^3$, we define $N(t,\theta)$ by
\begin{align*}
N(t,\theta) = B_{0,k}(t)(\theta-\theta_k^3)^2 + C_{0,k}(t)(\theta-\theta_k^3) + N(t,\theta_k^3),
\end{align*}
with 
$\theta_k^3B_{0,k}(t) =  \tfrac{1}{2}\pa_{\theta}N(t,\theta_{k}^3)$, $C_{0,k}(t) = \pa_{\theta}N(t,\theta_k^3)$ 
and analogously for $\pi - \theta_k^3\leq \theta \leq \pi$. We finally reflect evenly across the $x$-axis, to define $N(t,\theta)$ for all $\theta$.
\end{defn}
The coefficients $A_{j,k}(t)$, $B_{j,k}(t)$ and $C_{j,k}(t)$ are chosen so that $N(t,\theta_{j,k}^{+}) = N(t,\theta_{j,k}^{-})$ , and that $\pa_{\theta}N(t,\theta)$ matches at $\theta = \theta_{j,k}^{\pm}$ with the values coming from Definition \ref{defn:N2}.  It therefore ensures that for $\theta_{j,k}^{-} \leq \theta \leq \theta_{j,k}^{+}$, we have
\begin{align*}
N(t_k,\theta) = t_k\delta_k, \qquad  N(t_{k-1}e^{-M_1},\theta) = e^{-M_1}t_{k-1}\delta_{k-1}, \qquad
C^{-1} \delta_k \leq \int_{t_k}^{t_{k-1}e^{-M_1}} \frac{N(t,\theta)}{t^2} \ud t  \leq C\delta_k.
\end{align*}
Moreover, for this range of $\theta$
\begin{align*}
\left|\pa_{\theta}N(t,\theta)\right|  \leq  C\theta_k^{-2}N(t,\theta), \qquad 
\left|\pa^2_{\theta}N(t,\theta)\right|  \leq C \theta_k^{-5}N(t,\theta) 
\end{align*}
and so Property \ref{prop:4}) holds provided $\pa_tN(t,\theta)$ has the desired lower bound. For $1\leq j \leq \theta_k^{-1}-1$, we have
\begin{align} \label{eqn:inter2aa}
\pa_{t}N(t,\theta) = \left(\pa_tA_{j,k}(t) (\theta-\theta_{j,k}^{-})^3 + \pa_tB_{j,k}(t) (\theta-\theta_{j,k}^{-})^2 + \pa_tC_{j,k}(t)(\theta-\theta_{j,k}^{-})\right) + \pa_tN(t,\theta_{j,k}^{-}),
\end{align}
and we can bound the group of terms in parentheses in absolute value by
\begin{align} \label{eqn:inter2a}
\tilde{C}_1\theta_k^3\left(\left|\pa_t\pa_{\theta}N(t,\theta_{j,k}^{-})\right| + \left|\pa_t\pa_{\theta}N(t,\theta_{j,k}^{+})\right|\right). 
\end{align}
There is also the analogous bound for $j=0$. By a direct calculation, there exists an absolute constant $K$ such that for $k\geq K$ the terms in \eqref{eqn:inter2a} by $\tfrac{1}{2} \pa_tN(t,\theta_{j,k}^{-1})$, and so Property \ref{prop:4}) holds. 
\\
\\
The function $N(t,\theta)$ has now been defined for $t_k \leq t \leq t_{k-1}e^{-M_1}$, $0 \leq \theta \leq 2\pi$, so that Properties \ref{prop:1}), \ref{prop:2}), and \ref{prop:4}) hold (with constants independent of $k$ and $M_1$). Moreover, we have
\begin{align} \label{eqn:integral1}
\int_{t_k}^{t_{k-1}e^{-M_1}}\frac{N(t,\theta)}{t^2} \ud t = F_{k}(\theta) \delta_{k},
\end{align}
with $C^{-1} \leq F_{k}(\theta) \leq C$, for an absolute constant $C$ independent of $k$ and $M_1$. The bounds on $\pa_{\theta}N(t,\theta)$, $\pa^2_{\theta}N(t,\theta)$ from Property \ref{prop:4}) imply that 
\begin{align} \label{eqn:integral2}
\left|\pa_{\theta}F_{k}(\theta)\right| \leq C\theta_{k}^{-2}, \qquad \left|\pa^2_{\theta}F_{k}(\theta)\right| \leq C\theta_k^{-5}. 
\end{align}
\textbf{Definition of $N(t,\theta)$ for $t_{k-1}e^{-M_1} \leq t \leq t_{k-1}$}:
\\
\\
For $t_{k-1}e^{-M_1} \leq t \leq t_{k-1}$ and $0 \leq \theta \leq 2\pi$, we define $N(t,\theta)$ to ensure that Property \ref{prop:3}) holds (while maintaining the other properties).
\begin{defn} \label{defn:N5}
Define  $N(t,\theta)$ by  $N(t,\theta) = a_{k}(\theta)t_{k-1}^{-1}t^2 + b_{k}(\theta)t + c_{k}(\theta)t_{k-1}$ with the coefficients chosen so that
\begin{align*}
N(t_{k-1}e^{-M_1},\theta) = t_{k-1}e^{-M_1}\delta_{k-1}, \qquad N(t_{k-1},\theta) = t_{k-1}\delta_{k-1}, \qquad \int_{t_{k}}^{t_{k-1}} \frac{N(t,\theta)}{t^2} \ud t = \tilde{C}\delta_{k}.
\end{align*}
Here $\tilde{C}$ is an absolute constant  to be prescribed below.
\end{defn}
By choosing $M_1$ and $\tilde{C}$ to be sufficiently large (depending only on the constant $C$ from \eqref{eqn:integral2}), and solving for $a_k(\theta)$, $b_k(\theta)$, $c_k(\theta)$, Property \ref{prop:4}) holds for this range of $t$.
\\
\\
This now defines $N(t,\theta)$ for $\theta\in[0,2\pi]$, $0<t\leq t_K$, with Properties \ref{prop:1}) - \ref{prop:4}), and so by Claim \ref{claim:Main}, this proves Proposition \ref{prop:Dini4} and shows that Assumption \ref{assum:Main} does not automatically hold in $3$ and higher dimensions. 
\end{proof}

\section{Appendix} \label{sec:appendix}

In the Appendix we record the proofs of some of the technical propositions and lemmas needed to prove Theorems \ref{thm:Lipschitz} and \ref{thm:Phi}. 
\\
\\
\textbf{Proof of Theorem \ref{thm:convex}:} We first prove Theorem \ref{thm:convex} without Assumption \ref{assum:smooth} via an approximation argument. From Section \ref{sec:FH1}, the result of Theorem \ref{thm:convex} holds when $W$ is smooth and strictly geodesically convex, and $\gamma^{\pm}$ are $(n-2)$-dimensional submanifolds, smooth up to the boundary of $W$. The first step in the proof is therefore to approximate $W$ and $\gamma^{\pm}$.
\begin{prop} \label{prop:convex1}
Given $\eps, \delta>0$, there exist a smooth, strictly geodesically convex set $W_{\eps}\subset W$, and $(n-2)$-dimensional manifolds $\gamma^{\pm}_{\delta}\subset \bar{W}^{\pm}$, smooth up to the boundary of $\pa W$, such that
\begin{align*}
\emph{Haus}_{\mS^{n-1}}(W_{\eps},W)  <\eps, \qquad \emph{Haus}_{\mS^{n-1}}(\gamma^{\pm}_{\delta},\gamma^{\pm}) <\delta. 
\end{align*}
Here $\emph{Haus}_{\mS^{n-1}}(\cdot,\cdot)$ measures the Hausdorff distance between the sets in $\mS^{n-1}$. 
\end{prop}
\begin{proof}{Proposition \ref{prop:convex1}}
We first construct the set $W_{\eps}$. Since $W$ is a proper, convex subset of $\mS^{n-1}$, it must be contained in a hemisphere, say $\mS^{n-1}\cap\{x_{n}\geq0\}$. We use $W\subset \mS^{n-1}$ to form a $n$-dimensional cone $R$, with vertex at the origin. By the convexity of $W$, $R$ is a Euclidean convex subset of $\R^{n}$, and we set $K$ to be the $(n-1)$-dimensional cross section $K = R\cap\{x_{n} =1\}$.  We can therefore form a sequence $K_{\eps}\subset K$ of smooth, strictly convex, bounded sets, which converge to $K$ in Hausdorff distance on compact sets as $\eps$ tends to $0$, \cite{Sc}. Using $K_{\eps}$ to form the sequence of strictly Euclidean convex cones $R_{\eps}$, and then restricting to $\mS^{n-1}$, we obtain a sequence of strictly convex, smooth sets $W_{\eps} \subset W$, converging to $W$ in Hausdorff measure on $\mS^{n-1}$ as $\eps$ tends to $0$. 
\\
\\
To construct the smooth sets $\gamma^{\pm}_\delta$ we proceed as follows: Given $\eta>0$, since $w\in C^{\alpha}(W)$, the distance between $\{w > \eta\}$ and $\pa \{w > 0\}$ is strictly positive, and tends to zero as $\eta$ tends to $0$. We approximate $w$ from below by a sequence $w_n\in C^{\infty}(W)$, converging uniformly to $w$ on $W$.  Therefore, by choosing $n = n(\delta)$ sufficiently large and $\eta = \eta(\delta)>0$ sufficiently small and using Sard's lemma, we can construct a smooth set $\gamma^{+}_\delta = \{w_n > \eta\}$ contained in $\{w>0\}$, with the required properties, and analogously for $\gamma^{-}_{\delta}$. 
\end{proof}

Suppose first that one of $\gamma_{\delta}^{\pm}$ does not intersect $W_{\eps}$ for some $\delta$, $\eps>0$. Then, one of $W^{\pm}$ is contained in $W\backslash W_{\eps}$. If this holds for $\delta$ and $\eps>0$ sufficiently small, then since
\begin{align*}
\lim_{\eps\to0} \text{Vol}_{\mS^n}(W\backslash W_{\eps}) =0, \qquad  \lim_{\delta\to0}\text{Haus}_{\mS^n}(\gamma^{\pm}_{\delta},\gamma^{\pm}) =0,
\end{align*} 
one of $\alpha(W^{\pm})$ must be larger than $2$, and the estimate in Theorem \ref{thm:convex} holds automatically. Therefore, from now on we will assume that $\gamma_{\delta}^{\pm}$ both intersect $W_{\eps}$.
\\
\\
To complete the proof of the theorem, we will show that the approximation statement in Proposition \ref{prop:convex1} carries over to the Dirichlet-Neumann eigenvalues.
\begin{defn} \label{defn:approx-eigenvalues}
 Denoting $\gamma^{\pm}_{\eps,\delta}$ to be the part of $\gamma^{\pm}_{\delta}$ contained in $W_{\eps}$,  let $W^{\pm}_{\eps,\delta} \subset W_{\eps}$ be the set with boundary consisting of  $\gamma^{\pm}_{\eps,\delta}$ and the part of $\pa W_{\eps}$ contained in $W^{\pm}$. We also denote $W^{\pm}_{\eps}$ to be the sets $W_{\eps}\cap W^{\pm}$, which correspond to  $W^{\pm}_{\eps,\delta}$ with $\delta=0$.
\end{defn}
\begin{defn} \label{defn:DN-approx}
Let $\mu^{\pm}_{\eps,\delta}$ be the first eigenvalue for $({W}^{\pm}_{\eps,\delta},g_{\eps})$, with Dirichlet boundary conditions imposed on ${\gamma}^{\pm}_{\eps,\delta}$,and Neumann boundary conditions elsewhere. Let $\mu^{\pm}_{\eps}$ be the first eigenvalue for ${W}^{\pm}_{\eps}$, with Dirichlet boundary conditions imposed on $\gamma^{\pm}$, and Neumann boundary conditions on the rest of $\pa W^{\pm}_{\eps}$.
\end{defn}
The proof of Theorem \ref{thm:convex} is completed by the following lemma:
\begin{lem} \label{lem:eigenvalue-limit2}
The  limits $\lim_{\delta\to0} \mu^{\pm}_{\eps,\delta} = \mu_{\eps}^{\pm}$, $\limsup_{\eps\to0} \mu_{\eps}^{\pm} \leq \mu^{\pm}$ both hold. 
\end{lem}
\begin{proof}{Lemma \ref{lem:eigenvalue-limit2}}
Since $\lim_{\eps\to0} \text{Vol}_{\mS^n}(W\backslash W_{\eps}) = 0$, the second limit in the lemma follows immediately from Lemma \ref{lem:eigenvalue-lower}. To establish the first limit we proceed as follows. By the variational formulation of the first eigenvalue, we have
\begin{align*}
\mu^{\pm}_{\eps,\delta} = \inf_{\phi\in X^{\pm}_{\eps,\delta}} \frac{\int_{{W}^{\pm}_{\eps,\delta}} \left|\nabla_{g}\phi\right|^2  \ud \sigma_{g}}{\int_{{W}^{\pm}_{\eps,\delta}} \phi^2 \ud \sigma_{g}}.
\end{align*}
for $X^{\pm}_{\eps,\delta} = \{\phi \in C^{\infty}({W}^{\pm}_{\eps,\delta}) : \text{supp}(\phi) \cap \gamma^{\pm}_{\eps,\delta} =\emptyset , \phi\neq0\}.$ Define the set of functions $X^{\pm}_{\eps}$ by $X^{\pm}_{\eps} = \{\phi \in C^{\infty}({W}^{\pm}_{\eps}) : \text{supp}(\phi) \cap \gamma^{\pm} =\emptyset , \phi\neq0\}.$ Since $X^{\pm}_{\eps,\delta}\subset X^{\pm}_{\eps}$, we immediately obtain $\mu^{\pm}_{\eps,\delta}\geq\mu^{\pm}_{\eps}$. Moreover, given any $\phi \in X^{\pm}_{\eps}$, since the Hausdorff distance between $\gamma_{\eps,\delta}$ and $\gamma$ tends to $0$ as $\delta$ tends to $0$ (uniformly in $\eps>0$), we see that $\phi\in X^{\pm}_{\eps,\delta}$ for $\delta>0$ sufficiently small (uniformly in $\eps>0$). In particular, given $c>0$, we can choose $\delta$ sufficiently small so that $\mu^{\pm}_{\eps,\delta} \leq \mu^{\pm}_{\eps} +c$. Thus,  $\lim_{\delta\to0}\mu^{\pm}_{\eps,\delta} = \mu^{\pm}_{\eps},$ as required.
\end{proof}
\noindent
\textbf{Regularity of the spherical slices $V_t$:} In Section \ref{sec:FH2}, we required Lemma \ref{lem:graph} to prove properties of the spherical sections $V_{t}$ for those values of $t$ where $t^{-1}M_{\textbf{0}}(t)$ is sufficiently small. We now prove this lemma and its corollary, and we start by restating Lemma \ref{lem:graph}:
 \begin{lem} \label{lem:graph-a}
There exists a constant $c>0$, depending only on the Lipschitz norm of $\pa\Omega$, and an orientation of $\Omega$ with the following property: For $t\in S_{c}$, the part of the boundary $\pa\Omega \cap \left(B_{2t}\backslash B_{t/2}\right)$ can be written as the graph $x_n = g(x')$ of a convex function $g$, with Lipschitz constant depending only on that of $\pa\Omega$. Here we have written $x = (x',x_n)\in\R^n$.
 \end{lem}
 \begin{proof}{Lemma \ref{lem:graph-a}}
We fix an orientation of $\Omega$ by looking at the slice $V_1$: This slice $V_1\subset \mathbb{S}^{n-1}$ contains a geodesic ball of radius $r^*>0$, and we rotate $\Omega$ so that the center of this geodesic ball is at the north pole $(0,\ldots,0,1)\in\mathbb{S}^{n-1}$. Since the slices $V_t$ form a increasing sequencs as $t$ decreases, $V_t$ contains this geodesic ball for all $t<1$. Moreover, since $V_t$ is always contained in a hemisphere, this means that $V_t$ cannot contain the corresponding neighbourhood of the south pole. 
\\
\\
Now let $y\in \pa\Omega \cap \left(B_{2t}\backslash B_{t/2}\right)$, with outward pointing unit normal $\nu(y)$. To prove the lemma, we will show that the component of $\nu(y)$ in the $x_n$-direction is bounded above by $-c^*$, for a constant $c^*>0$ depending only on the Lipschitz constant of $\pa\Omega$. After a rotation in the $x'$-variables, we can assume that $y$ lies in the $(x_1x_n)$-plane. Writing $\nu(y) = (\nu_1,\ldots,\nu_n)$, we first obtain a lower bound on $\nu_1^2+\nu_n^2$ as follows: Since $\Omega$ is convex, the tangent plane at $y$ provides a supporting hyperplane for $\Omega$. Therefore, if $\nu_1^2+\nu_n^2$ is small relative to $r^*$, then this would contradict the spherical slices $V_t$ containing the geodesic ball of radius $r^*$ centred at the north pole.
\\
\\
Now that we have a lower bound on $\nu_1^2+\nu_n^2$, we use the upper bound on $s^{-1}M_{\textbf{0}}(s)$ for $s\in[\tfrac{1}{2}t,2t]$ to conclude the proof. Since $t\in S_{c}$, we have
\begin{align*}
|y|^{-1}\left(y_1\nu_1 + y_n\nu_n\right) = \frac{y}{|y|}\cdot \nu(y) \leq c,
\end{align*}
so that the vectors $\tfrac{y}{|y|}$ and $(\nu_1,\nu_n)$ are almost orthogonal in the $(x_1x_n)$-plane. Since $\tfrac{y}{|y|}$ cannot be within distance $r^*$ from the north or south pole, by taking $c$ sufficiently small depending on $r^*$, this provides a lower bound on $|\nu_n|$. Finally, $\nu_n$ must be negative, since otherwise this would again contradict $V_t$ containing the north pole.  
\end{proof}
 \begin{cor} \label{cor:graph-a}
For $c$ and the orientation of $\Omega$ as in Lemma \ref{lem:graph}, and for $s\in[\tfrac{1}{2}t,2t]$ with $t\in S_c$, the slices $V_s$ are star-shaped with respect to the north pole in $\mathbb{S}^{n-1}$. Moreover, there exists a neighbourhood around each $y\in\pa V_s$ on which $\pa V_s$ can be parameterized as submanifold $(x'(\tau),g(x'(\tau)))$ for $\tau$ in an open set in $\R^{n-2}$. The Lipschitz constant of this parameterization and the size of the neighbourhood can be taken to only depend on that of $\pa\Omega$.
\end{cor}
\begin{proof}{Corollary \ref{cor:graph-a}}
To establish the star-shaped property, it is sufficient to show that every great circle passing through the north pole intersects $\pa V_s$ at precisely two points. Without loss of generality, after a rotation in the $x'$ variables, we may take this great circle to be in $(x_1x_{n})$-plane, given by $\{x_2=x_3=\cdots=x_{n-1}=0\}\cap \mathbb{S}^{n-1}$. Let $y \in \pa V_s$ be on this great circle, with $y_1>0$. Then $sy\in \pa \Omega$, and writing $\nu(sy) = (\nu_1,\ldots, \nu_n)$, from Lemma \ref{lem:graph} we know that $\nu_n \leq - c_n$ for some constant $c_n >0$, and $|y|^{-1}\left(y_1\nu_1 + y_n\nu_n\right) \leq c$. Since as in the proof of Lemma \ref{lem:graph}, the tangent plane at $y$ is a supporting hyperplane for $\Omega$, for $c>0$ sufficiently small there exists no point $z$ in $\Omega$ on this great circle with $z_1>0$ and $z_n<y_n$. The analogous argument applies for the portion of the circle with $x_1<0$, and so this gives the star-shaped property.
\\
\\
We now establish the Lipschitz parameterization of $\pa V_s$. After a rotation in the $x'$-plane, let $y=(y_1,0,\ldots,0,s^{-1}g(sy'))\in \pa V_s$ with $y_1>0$. Here $g$ is the convex function from Lemma \ref{lem:graph}. Let ${v}(x_1,x_2)$ be given by
\begin{align*}
{v}(x') = x_1^2+\cdots +x_{n-1}^2 + s^{-2}g(sx')^2 - 1.
\end{align*}
Then, ${x}\in\pa V_s$ if and only if $x_n = s^{-1}g(sx')$ and $v(x') = 0$. In particular, $v(y') = 0$, and to establish the Lipschitz parameterization near this point, we will apply the implicit function theorem to the set $\{x':v(x') = 0\}$. We therefore, need to obtain a lower bound on 
\begin{align} \label{eqn:graph1}
\pa_{x_1}v(x') = 2x_1 + 2s^{-1}g(sx')(\pa_{x_1}g)(sx'),
\end{align} 
for $x'$ near $y'$. To obtain a lower bound on this quantity, we will use $M_{\textbf{0}}(t)$. For $|x| = s\in[\tfrac{1}{2}t,2t]$, we have
\begin{align*}
\nu(x)\cdot x = (1+|\nabla g(x')|^2)^{-1/2}(\pa_{x_1}g(x'),\ldots, \pa_{x_{n-1}}g(x'),-1)\cdot(x',g(x')) \leq cs,
\end{align*}
which we can rearrange as $g(x') \geq \nabla g(x')\cdot x' - cs(1+|\nabla g(x')|^2)^{1/2}.$ Inserting this in \eqref{eqn:graph1} we have
\begin{align*} 
\pa_{x_1}v(x') &\geq 2x_1 + 2\pa_{x_1}g(sx')\left(\nabla g(sx')\cdot x'\right)- 2c\pa_{x_1}g(sx')(1+|\nabla g(sx')|^2)^{1/2} \\
& = 2x_1 + 2x_1(\pa_{x_1}g(sx'))^2 + \text{ Error}.
\end{align*} 
Here the error terms consist of terms involving a factor of $x_2,\ldots,x_{n-1}$ or $c$. In particular, for $x'$ sufficiently near the point $y' = (y_1,0,\ldots,0)$ and for $c>0$ sufficiently small, the Error term is smaller than $x_1$, and we obtain the desired lower bound on $\pa_{x_1}v(x')$. 
\end{proof}
\textbf{Proof of Proposition \ref{prop:L1-harmonic}:} We end by proving Proposition \ref{prop:L1-harmonic}, which establishes a subdomain of $\Omega$ satisfying an $L^{\infty}$ harmonic measure estimate.
\begin{prop} \label{prop:L1-harmonic-a}
 There exist constants $c_1$, $C_1$, depending only on the Lipschitz norm of $\pa\Omega$,  such that for each ${x}_0\in \Omega\cap B_{1/2}$, we can form a  convex domain $\Omega_{{x}_{0}}\subset\Omega$, with the following properties:
\begin{enumerate}

\item[i)] The boundary of $\Omega_{{x}_{0}}$ consists of two parts $\pa\Omega_{{x}_{0},N}$ and $\pa\Omega_{{x}_{0},D}$. The first part is a (possibly empty) subset of $\pa\Omega$, and the second part ensures that $B_{c_1}({x}_{0}) \cap \Omega \subset \Omega_{{x}_{0}}$.

\item[ii)] Let $G$ be the Green's function for $\Omega_{{x}_{0}}$, with pole at ${x}_{0}$, with zero Dirichlet boundary conditions on $\pa\Omega_{{x}_{0},D}$, and zero Neumann boundary conditions (weakly) on $\pa\Omega_{{x}_{0},N}$. Then,
\begin{align*}
\norm{G}_{L^{\infty}\left(\Omega_{{x}_{0}}\backslash B_{c_1}({x}_0)\right)} \leq C_1, \qquad \norm{\nabla G}_{L^{\infty}\left(\pa\Omega_{{x}_{0},D}\right)} \leq C_1 .
\end{align*}

\end{enumerate}
\end{prop}
\begin{proof}{Proposition \ref{prop:L1-harmonic-a}} Let $y_0$ be the closest point to $x_0$ on $\pa\Omega$. After a translation, we set ${y}_0 = \textbf{0}$. Now, let $c>0$ be a small constant, as in the statement of Lemma \ref{lem:graph}, and fix $c_1$ with $2c_1\in S_{c}$. If $|x_0|\geq c_1$, we can take $\Omega_{{x}_0} = B_{c_1}(x_0)$. Otherwise, choose $z_0 = (0,\ldots,0,-|z_0|)$, with $|z_0| = \tfrac{1}{10}c_1$, and set $\Omega_{{x}_0} = B_{2c_1}(z_0)\cap\Omega$. In either, case part i) of the lemma is satisfied.
\\
\\
To obtain the estimates on $G$, we show that the Dirichlet, $\pa B_{2c_1}(z_0)\cap \Omega$, and Neumann, $B_{2c_1}(z_0)\cap\pa\Omega$, parts of the boundary meet at a strictly acute angle at each point of their intersection: For $y\in  \pa B_{2c_1}(z_0)\cap\pa\Omega$, we have
\begin{align*}
(y-z_0)\cdot \nu(y) = y\cdot\nu(y) - z_0\cdot \nu(y) \leq |y|M_{\textbf{0}}(|y|) - z_0\cdot\nu(y).
\end{align*}
We have $|y|^{-1}M_{\textbf{0}}(|y|) \leq c$, and by Lemma \ref{lem:graph}, there exists a constant $c^*$, independent of $c$, such that $z_0\cdot\nu(y) \geq c^*|z_0|$. Therefore, choosing $c<\tfrac{1}{20}c^*$, we have
\begin{align*}
(y-z_0)\cdot \nu(y) \leq - \tfrac{1}{20}c^*c_1.
\end{align*}
Since $|y-z_0|$ is comparable to $c_1$, this shows that $\pa\Omega$ and $B_{2c_1}(z_0)$ meet at a strictly acute angle at $y$. With this property, we can therefore use the maximum principle with a linear function to show that the Green's function $G(x)$ decays (at least) linearly as $x$ approaches each point $y\in\pa B_{2c_1}(z_0)\cap\pa\Omega$. Since near $\pa B_{2c_1}(z_0)\cap\pa\Omega$, the function $G$ is a harmonic function on a Lipschitz domain, this implies the desired estimate for $\nabla G$. 
\end{proof}

\end{document}